\documentclass[final,onefignum,onetabnum]{siamart251104}

\usepackage{microtype}
\usepackage{enumitem}

\newcommand{\bx}{\boldsymbol{x}}
\newcommand{\bs}{\boldsymbol{s}}
\newcommand{\blambda}{\boldsymbol{\lambda}}
\newcommand{\bu}{\boldsymbol{u}}

\newcommand{\bz}{\boldsymbol{z}}

\newcommand{\bq}{\boldsymbol{q}}

\newcommand{\bK}{\boldsymbol{K}}

\newcommand{\bF}{\boldsymbol{F}}
\newcommand{\bG}{\boldsymbol{G}}
\newcommand{\bd}{\boldsymbol{d}}
\newcommand{\bp}{\boldsymbol{p}}

\newcommand{\bLambda}{\boldsymbol{\Lambda}}

\newcommand{\bB}{\boldsymbol{B}}
\newcommand{\bC}{\boldsymbol{C}}

\newcommand{\bH}{\boldsymbol{H}}

\newcommand{\br}{\boldsymbol{r}}

\newcommand{\st}{\mathop{\text{\normalfont s.t.}}}
\newcommand{\diag}{\mathop{\text{\normalfont diag}}}
\newcommand{\rank}{\mathop{\text{\normalfont rank}}}

\newcommand{\cK}{\mathcal{K}}

\newcommand{\cI}{\mathcal{I}}
\newcommand{\cJ}{\mathcal{J}}

\usepackage{braket,amsfonts}

\usepackage{array}

\usepackage[caption=false]{subfig}
\usepackage{pgfplots}

\usepackage{amsmath}
\crefname{equation}{}{}
\newtheorem{assumption}{Assumption}[section]

\usepackage{cleveref}

\crefname{theorem}{Theorem}{Theorems}
\crefname{assumption}{Assumption}{Assumptions}

\newsiamremark{remark}{Remark}
\usepackage{algorithmic}

\usepackage{graphicx,epstopdf}

\usepackage{amsopn}

\usepackage{xspace}
\usepackage{bold-extra}
\usepackage[most]{tcolorbox}

\colorlet{texcscolor}{blue!50!black}
\colorlet{texemcolor}{red!70!black}
\colorlet{texpreamble}{red!70!black}
\colorlet{codebackground}{black!25!white!25}

\lstdefinestyle{siamlatex}{%
  style=tcblatex,
  texcsstyle=*\color{texcscolor},
  texcsstyle=[2]\color{texemcolor},
  keywordstyle=[2]\color{texemcolor},
  moretexcs={cref,Cref,maketitle,mathcal,text,headers,email,url},
}

\tcbset{%
  colframe=black!75!white!75,
  coltitle=white,
  colback=codebackground, % bottom/left side
  colbacklower=white, % top/right side
  fonttitle=\bfseries,
  arc=0pt,outer arc=0pt,
  top=1pt,bottom=1pt,left=1mm,right=1mm,middle=1mm,boxsep=1mm,
  leftrule=0.3mm,rightrule=0.3mm,toprule=0.3mm,bottomrule=0.3mm,
  listing options={style=siamlatex}
}

\newtcblisting[use counter=example]{example}[2][]{%
  title={Example~\thetcbcounter: #2},#1}

\newtcbinputlisting[use counter=example]{\examplefile}[3][]{%
  title={Example~\thetcbcounter: #2},listing file={#3},#1}

\DeclareTotalTCBox{\code}{ v O{} }
{ %fontupper=\ttfamily\color{texemcolor},
  fontupper=\ttfamily\color{black},
  nobeforeafter,
  tcbox raise base,
  colback=codebackground,colframe=white,
  top=0pt,bottom=0pt,left=0mm,right=0mm,
  leftrule=0pt,rightrule=0pt,toprule=0mm,bottomrule=0mm,
  boxsep=0.5mm,
  #2}{#1}

\allowdisplaybreaks
\usepackage{mathtools}
\usepackage{parskip}

\makeatletter
\newif\ifinappendix
\inappendixfalse

\crefformat{section}{%
  \ifinappendix
    Appendix~#2#1#3%
  \else
    Section~#2#1#3%
  \fi
}
\Crefformat{section}{%
  \ifinappendix
    Appendix~#2#1#3%
  \else
    Section~#2#1#3%
  \fi
}
\makeatother

\usepackage[titletoc]{appendix}

\makeatletter
\AddToHook{cmd/appendix/before}{\def\cref@section@alias{appendix}}
\makeatother
\let\oldoverset\overset
\renewcommand{\overset}[2]{%
  \oldoverset{\raisebox{0.2ex}{\scriptsize #1}}{#2}}
\setenumerate{topsep=1pt,parsep=1pt,partopsep=1pt,leftmargin=*}

\newcommand{\WT}[1]{{\color{black} #1}}
\title{
  Improved Approximation Bounds for Moore-Penrose Inverses of Banded Matrices with Applications to Continuous-Time Linear Quadratic Control\thanks{
    This material is based upon work supported by the U.S. Department of Energy, Office of Science, Office of Advanced Scientific Computing Research (ASCR) under Contract DE-AC02-06CH11347.
    }
}

\author{
  Sungho Shin\thanks{Department of Chemical Engineering, Massachusetts Institute of Technology (\url{sushin@mit.edu}).
    Majority of this work is performed while the author was a postdoctoral researcher at Argonne National Laboratory.}
  \and
  Wallace Gian Yion Tan \thanks{Department of Chemical Engineering, Massachusetts Institute of Technology (\url{wtgy@mit.edu}).}
  \and
  Mihai Anitescu\thanks{Mathematics and Computer Science Division, Argonne National Laboratory, Lemont, IL (\url{anitescu@mcs.anl.gov}) and Department of Statistics, University of Chicago, Chicago, IL.}
}

\date{}

\begin{document}
\maketitle

%% ------------------------------------------------------------------
%% ABSTRACT
%% ------------------------------------------------------------------
\begin{abstract}
  We present improved approximation bounds for the Moore-Penrose inverses of banded matrices, where the bandedness is induced by a metric on the index set. We show that the pseudoinverse of a banded matrix can be approximated by another banded matrix, and the error of approximation is exponentially small in the ratio of the bandwidth of the approximation to that of the original matrix. An intuitive corollary can be obtained: the off-diagonal blocks of the pseudoinverse decay exponentially with the distance between the node sets associated with row and column indices, on the given metric space. Our bounds are expressed in terms of the bound of singular values of the system. For saddle point systems, commonly encountered in optimization, we provide the bounds of singular values associated under standard regularity conditions. Remarkably, our bounds improve previously reported ones and allow us to establish a perturbation bound for continuous-domain optimal control problems by analyzing the asymptotic limit of their finite difference discretization, which has been challenging with previously reported bounds. 
\end{abstract}

% \begin{MSCcodes}
%   36B20; % perturbations
%   05C50; % graphs and matrices
% \end{MSCcodes}
%% ------------------------------------------------------------------
%% END HEADER 
%% ------------------------------------------------------------------

\section{Introduction}\label{sec:intro}
We study an approximation of the Moore-Penrose inverse of a matrix $A\in \mathbb{R}^{m\times n}$ satisfying:
\begin{align}\label{eqn:glin}
  A[I_i,J_j]=0,\quad \forall i,j\in V \text{ such that } d(i,j)> \overline{\kappa},
\end{align}
where
$(V\subseteq \mathbb{Z}_{>0},d:V\times V \rightarrow \mathbb{R})$ is a finite metric space;
$\cI:=\{I_i\}_{i\in V}$ and $\cJ:=\{J_i\}_{i\in V}$ are partitions of $\mathbb{Z}_{[1,n]}$ and $\mathbb{Z}_{[1,m]}$, respectively;
and $\overline{\kappa}\in\mathbb{Z}_{\geq 0}$.
Here,
$\mathbb{R}$ and $\mathbb{Z}$ denote the set of real numbers and the set of integers, respectively;
$\mathbb{Z}_{[a,b]}:=\mathbb{Z}\cap [a,b]$;
$\mathbb{Z}_{\geq 0}:=\mathbb{Z}_{[0,\infty)}$; 
$\mathbb{Z}_{> 0}:=\mathbb{Z}_{(0,\infty)}$;
and $A[I_i,J_j]:=\{A[i',j']\}_{i'\in I_i,j'\in J_j}$, where $A[i',j']$ is the $(i',j')$-th component of $A$.
The constant $\overline{\kappa}$ in \eqref{eqn:glin}, which represents the bandwidth of matrix $A$ associated with the index set partitions $(\cI,\cJ)$ and metric space $(V\subseteq \mathbb{Z}_{>0},d:V\times V \rightarrow \mathbb{R})$, will be simply referred to as the \emph{bandwidth} of $A$.

In the main theorem, we show that for all $\kappa\in \mathbb{Z}_{>0}$, if the nonzero singular values $\sigma_{>0}(A)$ of $A$ are within $[a,b]\subseteq \mathbb{R}_{>0}$, the Moore-Penrose inverse $A^+\in\mathbb{R}^{n\times m}$---a unique matrix satisfying $AA^+ A = A$, $ A^+A A^+ = A^+$, $ (AA^+)^\top = AA^+$, and $ (A^+A)^\top =A^+A$---admits an exponentially accurate $\kappa$-banded approximation $\widetilde{A}\in \mathbb{R}^{n\times m}$:
\begin{align*}
  \widetilde{A}[J_i,I_j]= 0\quad \forall i,j\in V \text{ such that } d(i,j)> \kappa
\end{align*}
and
\begin{align}\label{eqn:approx-bound}
\left\| A^+- \widetilde{A} \right\|_2& \leq \frac{4\left(b+a\right)^{3/2}}{ab(b-a)^{1/2}}\left(\frac{b-a}{b+a}\right)^{\left\lceil \frac{\kappa/\overline{\kappa} -1}{2}\right\rceil}.
\end{align}
\WT{where $\left\lceil\cdot\right\rceil$ is the ceiling operator.}
That is, by increasing the bandwidth upper bound of the approximation $\widetilde{A}$, the approximation becomes exponentially accurate.
A tighter bound is provided for the symmetric positive semi-definite case $A\succeq 0$.
An important corollary can be drawn from the above result: the off-diagonal blocks of $A^+$ exponentially decay with the distance between the node sets associated with row and column indices on $(V,d)$.

The exponentially accurate banded approximation finds applications in diverse areas.
It can be used as a banded preconditioner \cite{johnsonPolynomialPreconditionersConjugate1983},
Furthermore, the banded approximation bounds have been used to analyze the convergence rate of decomposition algorithms for large-scale optimization problems \cite{na_convergence_2022}.
Recently, these bounds are applied to show the near-optimality of distributed control \cite{shin_near-optimal_2023} and stochastic predictive control \cite{shin_near-optimal_2023-1}.
Lastly, these bounds can be used to compute approximations of dense inverses or Schur complements, which often have significance in uncertainty quantification \cite{chen2023scalable}. 

The existing literature \cite{shin2022exponential, demko1984decay} suggests that for square, non-singular cases, the off-diagonal blocks of the inverse exhibit an exponentially decaying bound relative to the distance between the blocks over the given metric space. In \cite{shin2022exponential}, the concept of a graph-induced banded system is introduced and an exponential perturbation bound is established. In \cite{demko1984decay}, the decay of the inverse of positive definite, banded matrices was studied.
In \cite{bickel2012approximating}, the approximation bound of the inverse of banded linear operators has been studied, and the results are applied to analyze the mixing property of the covariance matrix of Gaussian processes.
The related works \cite{shin2022exponential,demko1984decay,bickel2012approximating} are all based on polynomial approximation strategy \cite{rivlin_introduction_1981}, and thus, with a slight modification, these results can be extended to indefinite banded systems indexed by metric spaces, as outlined in \cref{prop:demko,prop:shin}. 
The bounds derived from the existing approaches (after suitable modifications) are shown below:
\begin{align}
  % \left\|A^+ - \widetilde{A}\right\|_2
  % &\leq \frac{4\left(b^2-a^2\right)^{1/2}}{ab}\left(\frac{b-a}{b+a}\right)^{\left\lceil \frac{\kappa/\overline{\kappa}-3}{2}\right\rceil}\tag{\cref{cor:main}; our result} \\
  \left\|A^+ - \widetilde{A}\right\|_2
  &\leq \frac{(a+b)^2}{2a^2b}\left(\frac{b-a}{b+a}\right)^{\left\lceil \frac{\kappa/\overline{\kappa}-1}{2}\right\rceil}\tag{\cref{prop:demko}; based on \cite{demko1984decay}}\\
  \left\|A^+ - \widetilde{A}\right\|_2
  &\leq \frac{b}{a^2}\left(\frac{b^2-a^2}{b^2+a^2}\right)^{\left\lceil \frac{\kappa/\overline{\kappa}-1}{2}\right\rceil}, \tag{\cref{prop:shin}; based on \cite{shin2022exponential,bickel2012approximating}}
\end{align}
where the same bandwidth upper bound is imposed to the approximation $\widetilde{A}$.
By comparing \cref{eqn:approx-bound} with the above equation, one can see that, in the case \WT{where $b$ is sufficiently large}, indeed our bound is tighter than the existing results.

The proof technique used in our work is similar to previous results \cite{demko1984decay,shin2022exponential,bickel2012approximating}, in that they all utilize a certain polynomial approximation $p(x)$ of $1/x$ to derive an error bound of $A^+ - p(A)$ for a matrix $A$. However, this approach only applies to positive definite $A$ and cannot handle indefinite, rank-deficient $A$. To address this limitation, we utilize an odd polynomial (i.e., $p(x) = -p(-x)$) along with the singular value decomposition of $A$. This approach enables approximating $A^+$ with a polynomial-like expression of $A$ and $A^\top$. This approach allows us to find a tighter approximation bound for indefinite, or even rectangular systems.

One of the practical limitations of the existing results \cite{demko1984decay,shin2022exponential} (and the modified versions in \cref{prop:demko,prop:shin})
is that they do not provide tight bounds for systems obtained by the finite-difference discretization of continuous-domain problems, like optimal control.
With the existing results, the error bounds diverge as the discretization mesh size decreases.
Our results overcome this limitation by establishing error bounds that can remain uniformly bounded as the discretization mesh size tends to zero.
We demonstrate this by analyzing the solution mapping for a discretized optimal control problem and showing that the perturbation bounds are uniform in the mesh size.
This further allows us to characterize the solution mapping in the continuous domain using the asymptotic limit of the solution mapping for the discretized problem. 

\WT{
Moreover, existing results only apply to instances where the metric is induced by geometric distances between nodes or graph adjacency structures. Our results generalize both cases to a much wider variety of settings by allowing for \textit{any} metric on the index set, as long as the bandedness property in \cref{eqn:glin} is satisfied. This flexibility is important as notably, in PDE discretizations, the geometric distance between two nodes may differ from the graph-induced distance.  
}

\paragraph{Contributions}
This paper has three main contributions. First, we generalize the concept of the bandedness to arbitrary metric spaces, allowing for more versatile applications of our main approximation result. Second, we provide an improved approximation bound for the Moore-Penrose inverse of banded matrices, which improves the results of \cite{shin2022exponential,demko1984decay}. The tighter bound allows analyzing the discretization of infinite-dimensional linear systems. For the third contribution, we apply our result to analyze the perturbation bound of continuous-time linear-quadratic optimal control problems, arguably one of the most studied linear system problems. We show that the solution mappings of this problem exhibit uniform exponential decay. This result establishes the connection between the exponential decay in discrete spaces with the counterparts in continuous-domain analysis \cite{grune2020exponential,gottlichPerturbationsPDEconstrainedOptimal2024}

\paragraph{Organization}

The remainder of the paper is organized as follows. 
In \Cref{sec:main}, we present our main results, namely the exponential bound for the banded approximations.
In \Cref{sec:kkt}, we demonstrate that the saddle point systems, derived from optimization problems, have bounded singular values when satisfying standard regularity conditions, such as linear independence constraint qualifications, second-order sufficiency conditions.  
\WT{In \Cref{sec:control}, we apply our results to analyze the optimal control problems.
Numerical verification of the main results are provided in \Cref{sec:num}.}
Finally, \Cref{sec:concl} presents the conclusions and future outlook.
% Proofs that are not critical for understanding the main point of the paper are separately provided in \Cref{apx:proofs}.

\section{Main Results}\label{sec:main}
In this section, we present our main results, which are the exponential bounds for the banded approximations discussed in \Cref{sec:intro}. We state the main theorem and provide a proof using polynomial or odd polynomial approximations of $1/x$. The main challenge lies in obtaining tight bounds for these polynomial approximations, which we achieve by using Chebyshev polynomials.
Finally, we compare our findings with existing bounds in the literature.

\subsection{Banded Matrices Indexed by Finite Metric Spaces}
In this section, we formally define the metric space and banded systems indexed by metric spaces.

\begin{definition}\label{def:metric}
  A metric space is a pair $(V,d)$ consisting of a set $V$ and a function $d:V\times V\rightarrow \mathbb{R}$ that satisfies the following properties:
  \begin{enumerate}[label=(\alph*)]
  \item\label{enum:metric-a} $d(i,j)\geq 0$ for all $i,j\in V$.
  \item\label{enum:metric-b} $d(i,j)=0$ if and only if $i=j$.
  \item\label{enum:metric-c} $d(i,j)=d(j,i)$ for all $i,j\in V$.
  \item\label{enum:metric-d} $d(i,j)\leq d(i,k)+d(k,j)$ for all $i,j,k \in V$.
  \end{enumerate}
  \WT{Furthermore, if $V$ is finite, the metric space $(V,d)$ is called a finite metric space.}
\end{definition}

\begin{remark}
  One example of a finite-dimensional metric space is $(V,d_G)$, where $G = (V,E)$ is a graph, $V$ is the node set, $E \subseteq \{\{i,j\} : i,j \in V\}$ is the edge set, and $d_G(i,j)$ is the geodesic distance induced by the graph $G$. Here, $d_G(i,j)$ is defined as the length of the shortest path from $i$ to $j$, where a path from $i \in V$ to $j \in V$ is a sequence of edges $\{\{i_k,i_{k+1}\}\}_{k=0}^{m-1}$ such that $i_0 = i$ and $i_m = j$. One can confirm that the conditions in \cref{def:metric} are satisfied (see \cite{shin2022exponential} for the complete proof).
\end{remark}

We now formally define the banded matrices indexed by finite metric spaces.
\begin{definition}\label{def:banded}
  Consider $A\in \mathbb{R}^{m\times n}$, a finite metric space $(V,d)$, partitions $\cI:=\{I_i\}_{i\in V}$ and $\cJ:=\{J_j\}_{j\in V}$ of $\mathbb{Z}_{[1,n]}$ and $\mathbb{Z}_{[1,m]}$, and $\overline\kappa\in\mathbb{Z}_{>0}$. We say $A$ is $(\overline\kappa,V,d,\cI,\cJ)$-banded if \cref{eqn:glin} is satisfied.
\end{definition}
When the corresponding metric space $(V,d)$ and partitions $\cI$ and $\cJ$ are clear from the context, we abbreviate the notation and say $A$ is $\overline\kappa$-banded. Moreover, we say $\overline\kappa$ is the bandwidth of $A$ if $\overline\kappa$ is the smallest integer such that $A$ is $\overline\kappa$-banded.

\WT{Note that the same sparsity structure for $A$ can be induced by multiple distinct metrics on the index set. Although the appropriate metric may be inferred from the particular context, such as in graph-structured problems, or be purely based on the sparsity pattern of the system (e.g., see \cite[remark 2.2]{shin2022exponential}), \cref{def:banded} can be applied independently of the context in which the bandedness arises.} 

%Note that the same matrix $A$ can be indexed by different metric spaces. One may employ metric space that arise from particular context, or may construct the metric space purely based on the sparsity pattern of the system (e.g., see \cite[remark 2.2]{shin2022exponential}).

Now we study the basic properties of the banded matrices indexed by metric spaces.
\begin{proposition}\label{prop:basic}
  Suppose that $A\in \mathbb{R}^{m\times n}$, $B\in \mathbb{R}^{m\times n}$, $C\in \mathbb{R}^{n\times \ell}$, are $(\kappa_A,V,d,\cI,\cJ)$,  $(\kappa_B,V,d,\cI,\cJ)$, and $(\kappa_C,V,d,\cJ,\cK)$-banded, respectively. Then,
  \begin{enumerate}[label=(\alph*)]
  \item\label{enum:basic-a} $A^\top$ is $(\kappa_A,V,d,\cJ,\cI)$-banded
  \item\label{enum:basic-b} $A+B$ is $(\max(\kappa_A,\kappa_B),V,d,\cI,\cJ)$-banded
  \item\label{enum:basic-c} $AC$ is $(\kappa_A+\kappa_C,V,d,\cI,\cK)$-banded.
  \end{enumerate}
\end{proposition}
\begin{proof}
  Throughout the proof, we will let $\cI:=\{I_i\}_{i\in V}$, $\cJ:=\{J_i\}_{i\in V}$, and $\cK:=\{K_i\}_{i\in V}$.

We first prove \ref{enum:basic-a}. We first observe that for all $i,j\in V$ with $d(i,j)>\kappa_A$, we have $d(j,i) > \kappa_A$ due to \cref{def:metric}\ref{enum:metric-c}, and thus, we have $$(A^\top)[J_i,I_j]=A[I_j,J_i]^\top=0.$$ Thus, $A^\top$ is $\kappa$-banded.

Next, we prove \ref{enum:basic-b}. We observe that if $d(i,j)>\max(\kappa_A,\kappa_B)$, we have that both $A[I_i,J_j]$ and $B[I_i,J_j]$ are both zero. Thus, $$(A+B)[I_i,J_j]=0$$ as well. Therefore, $A+B$ is $\kappa$-banded.

Lastly, we prove \ref{enum:basic-c}. We have that
For each $i,j,k\in V$ with $d(i,j) > \kappa_A + \kappa_C$, due to \cref{def:metric}\ref{enum:metric-a} and \ref{enum:metric-d}, we either have $d(i,k) > \kappa_A$ or $d(k,j) > \kappa_C$. This implies that $$(AC)[I_i,J_j] = \sum_{k\in V}A[I_i,K_k]C[K_k,J_j] = 0,\quad \forall i,j\in V\text{ such that }d(i,j)>\kappa_A+\kappa_C.$$
Therefore, $AC$ is $(\kappa_A+\kappa_C)$-banded.

\end{proof}

\cref{prop:basic} shows that the bandedness property is preserved under transposition, addition, and multiplication. Thus, applying polynomial operations allows for the matrix bandwidth to be bounded from above. This rationale justifies the use of a polynomial approximation strategy for the pseudoinverse.

\subsection{Main Theorem}
We are now ready to state our main theorem.
\begin{theorem}\label{thm:main}
  Consider $A\in \mathbb{R}^{m\times n}$, a finite metric space $(V,d)$, partitions $\cI:=\{I_i\}_{i\in V}$ and $\cJ:=\{J_j\}_{j\in V}$ of $\mathbb{Z}_{[1,n]}$ and $\mathbb{Z}_{[1,m]}$, and $\overline\kappa\in\mathbb{Z}_{>0}$. Suppose that $A$ is $(\overline{\kappa},V,d,\cI,\cJ)$-banded, and $\sigma_{>0}(A)\subseteq [a,b]\subseteq \mathbb{R}_{>0}$. For all $\kappa\in \mathbb{Z}_{>0}$, there exists $({\kappa},V,d,\cJ,\cI)$-banded $\widetilde{A}\in \mathbb{R}^{n\times m}$ such that:
  \begin{align}\label{eqn:main}
    \left\|A^+ - \widetilde{A}\right\|_2\leq
    f_A(\kappa/\overline{\kappa},a,b)
  \end{align}
  where
  \begin{align}\label{eqn:f}
    f_A(\omega,a,b)
    &:=
      \begin{cases}
        f_1(\omega,a,b):=\dfrac{(\sqrt{b}+\sqrt{a})^2}{2a b}\left(\dfrac{\sqrt{b}-\sqrt{a}}{\sqrt{b}+\sqrt{a}}\right)^{\left\lceil \omega\right\rceil},
        &\text{if $A\succeq 0$}.\\
        f_2(\omega,a,b):=\dfrac{4\left(b+a\right)^{3/2}}{ab(b-a)^{1/2}}\left(\dfrac{b-a}{b+a}\right)^{\left\lceil \frac{\omega -1}{2}\right\rceil}&\text{otherwise.}
      \end{cases}
  \end{align}
\end{theorem}

% We have implicitly assumed that $m=n$ when $A\succeq 0$.
\WT{The proof of \cref{thm:main} relies on \cref{lem:decay} for bounding the approximation error and \cref{lem:polynomial} for constructing an appropriate polynomial-like matrix approximation. The proofs of both lemmas will be given later in this section.}

The theorem states that banded systems have exponentially accurate banded approximations, with the error bound decaying exponentially as the bandwidth upper bound of the approximation increases. The bounds depend only on the singular value bounds and the bandwidth upper bounds of the original matrix and the approximate Moore-Penrose inverse. As long as the singular value bounds and the bandwidth upper bounds are uniformly bounded, the bounds will remain uniform even as the system dimension increases. For both the general and positive semi-definite case, the constant factors $(\sqrt{b}+\sqrt{a})^{2}/2ab$ and $4(b+a)^{3/2}/ab(b-a)^{1/2}$ are  $O(1/a)$, which is reasonable since $\|A^+\|_2$ is always smaller than the smallest singular value of $A$. One can also observe that the bounds that can be applied to $A\succeq 0$ case are tighter than the general one.

\WT{
\begin{remark}\label{rmk:neumann}
  Classical Neumann Series approaches provide polynomial approximation for $A^{-1}$ by using the truncation of the Neumann Series: $A^{-1} = (I-E)^{-1} =  \sum_{k=0}^\infty E^k $, where $A = I - E$, and the spectral radius of $E$ is assumed to be strictly less than $1$. This essentially applies polynomial approximation of $1/x$ over $1+(1-x) + (1-x)^2 + \cdots$. Our approach aims to apply tighter polynomial approximations of $1/x$ over $[a,b]$ directly, which leads to improved bounds.
\end{remark}
}

\cref{thm:main} is established by leveraging polynomial approximations of $1/x$ over $[a,b]\subseteq \mathbb{R}_{> 0}$.
The following lemma provides the connection between polynomial approximations and the banded approximation bounds.
Here, we denote by $P_n$ the set of $n$-th degree polynomials.

\begin{lemma}\label{lem:decay}
   \WT{Suppose that $A \in \mathbb{R}^{m \times n}$ is $(\overline{\kappa},V,d,\cI,\cJ)$-banded, $\sigma_{>0}(A)\subseteq [a,b]\subseteq \mathbb{R}_{>0}$}, and that there exist a sequence $\{p_{2n+1}(\cdot)\in P_{2n+1}\}_{n=0}^\infty$ of odd polynomials and $g(\cdot)$ such that
  \begin{align}\label{eqn:polbound}
    \left|1/x - p_{2n+1}(x) \right|\leq g(n), \quad \forall n\in\mathbb{Z}_{\geq 0},\; x\in[a,b].
  \end{align}
  Then, for all $\kappa\in\mathbb{Z}_{>0}$, there exists a $({\kappa},V,d,\cJ,\cI)$-banded $\widetilde{A}$ such that
  \begin{align}\label{eqn:decay}
    \left\|A^{+} - \widetilde{A}\right\|_2 \leq g\left(\left\lceil\frac{\kappa/\overline{\kappa}-3}{2}\right\rceil\right).
  \end{align}
  Furthermore, suppose that $A\succeq 0$, and there exist a sequence $\{q_n(\cdot)\in P_n \}_{n=0}^\infty$ of polynomials and $h(\cdot)$ such that
  \begin{align}\label{eqn:polbound-psd}
    \left|1/x - q_{n}(x) \right|\leq h(n), \quad \forall n\in\mathbb{Z}_{\geq 0},\; x\in[a,b].
  \end{align}
  Then, for all $\kappa\in\mathbb{Z}_{>0}$, there exists a $({\kappa},V,d,\cJ,\cI)$-banded $\widetilde{A}$ such that
  \begin{align}\label{eqn:decay-psd}
    \left\|A^{+} - \widetilde{A}\right\|_2 \leq h\left(\left\lceil\kappa/\overline\kappa-1\right\rceil\right).
  \end{align}
\end{lemma} 

\begin{proof}
  We first prove the $A\succeq 0$ case. \WT{We explicitly introduce $\{q_n(\cdot)\in P_n \}_{n=0}^\infty$ satisfying \cref{eqn:polbound-psd} as follows:
  \begin{align}\label{eqn:qn}
    q_{n}(x)=\sum_{j=0}^n b_{j,n}x^{j}, \quad n\in\mathbb{Z}_{\geq 0},
  \end{align}
where $b_{j,n}\in\mathbb{R}$ is the coefficient of $x^j$ in $q_n(x)$. Furthermore, we let} $$\widetilde{A}:= q_{\left\lceil\kappa/\overline\kappa-1\right\rceil}(A).$$ We observe that due to \cref{prop:basic}, the bandwidth of $\widetilde{A}$ is not greater than $\overline{\kappa}\left\lceil\kappa/\overline\kappa-1\right\rceil$. By using the fact that $\left\lceil x\right\rceil < x+1$, we have
  \begin{align*}
    \overline{\kappa}\left\lceil\kappa/\overline\kappa-1\right\rceil
     < \kappa.
  \end{align*}
  Thus, $\widetilde{A}$ satisfies the bandwidth constraint.

  We now recall that $A$ admits an SVD of the following form: $$A=\Phi\Lambda \Phi^\top,$$ where $\Lambda:=\diag(\lambda_1,\cdots,\lambda_{\rank(A)})$, $\sigma_{>0}(A)=\{\lambda_1,\cdots,\lambda_{\rank(A)}\}$, and $\Phi$ is a matrix with orthonormal columns. Now, we observe that\WT{
  \begin{align}\label{eqn:qn-1}
    \begin{aligned}[t]
      \widetilde{A}&=q_{\left\lceil\kappa/\overline\kappa-1\right\rceil}(A)\\
      &\stackrel{\mathclap{\cref{eqn:qn}}}{=} \sum_{j=0}^{\left\lceil\kappa/\overline\kappa-1\right\rceil} b_{j,\left\lceil\kappa/\overline\kappa-1\right\rceil} (\Phi \Lambda \Phi^\top)^{j}\\
      &= \Phi \left(\sum_{j=0}^{\left\lceil\kappa/\overline\kappa-1\right\rceil} b_{j,\left\lceil\kappa/\overline\kappa-1\right\rceil} \Lambda^j \right) \Phi^\top\\
      &\stackrel{\mathclap{\cref{eqn:qn}}}{=} \Phi q_{\left\lceil\kappa/\overline\kappa-1\right\rceil}(\Lambda) \Phi^\top,
    \end{aligned}
  \end{align}}
where the third equality follows from the orthonormality of the columns of $\Phi$.
We now recall that the unique Moore-Penrose inverse can be expressed by $A^+=\Phi\Lambda^{-1}\Phi^\top$, and observe the following.
  \begin{align*} 
    &\|A^+-q_{\left\lceil\kappa/\overline\kappa-1\right\rceil}(A)\|_2 \\
    &\stackrel{\mathclap{\cref{eqn:qn-1}}}{=} \|\Phi(\Lambda^{-1}-q_{\left\lceil\kappa/\overline\kappa-1\right\rceil}(\Lambda))\Phi^\top\|_2\\
    &= \|\Lambda^{-1}-q_{\left\lceil\kappa/\overline\kappa-1\right\rceil}(\Lambda)\|_2\\ 
    &= \WT{
      \left\|\begin{bmatrix}
        (1/\lambda_1)-\sum_{j=0}^{\left\lceil\kappa/\overline\kappa-1\right\rceil}b_{j,\left\lceil\kappa/\overline\kappa-1\right\rceil}   \lambda^{j}_1\\
      &\ddots\\
      && (1/\lambda_{\rank(A)})-\sum_{j=0}^{\left\lceil\kappa/\overline\kappa-1\right\rceil}b_{j,\left\lceil\kappa/\overline\kappa-1\right\rceil}   \lambda^{j}_{\rank(A)}
    \end{bmatrix}\right\|_2}\\
    &\stackrel{\mathclap{\cref{eqn:qn}}}{=} \max_{\lambda \in \sigma_{>0}(A)} | 1/\lambda - q_{\left\lceil\kappa/\overline\kappa-1\right\rceil}(\lambda)|\\
    &\stackrel{\mathclap{\cref{eqn:polbound-psd}}}{\leq} h(\left\lceil\kappa/\overline\kappa-1\right\rceil),
  \end{align*}
  where the second inequality follows from the orthonormality of the columns of $\Phi$, and the third inequality follows from the diagonality of $\Lambda^{-1}$ and $q_{\lceil\kappa/\overline{\kappa}-1\rceil}(\Lambda)$.

  Now we prove the general case. Let
  \begin{align}\label{eqn:pn}
    p_{2n+1}(x)=\sum_{j=0}^n a_{2j+1}x^{2j+1},\qquad P_{2n+1}(A) := \sum_{j=0}^n a_{2j+1} A (A^\top A)^{j}
  \end{align}
  for all $n\in\mathbb{Z}_{\geq 0}$,
  and we choose $\widetilde{A}:= P_{2n_0+1}(A)$, where $n_0:=\left\lceil \frac{\kappa/\overline{\kappa}-3}{2}\right\rceil$.
  First, we verify that $\widetilde{A}$ satisfies the bandwidth constraint.
  \cref{prop:basic} implies that the bandwidth of $P_{2n_0+1}(A)$ is not greater than $(2n_0+1)\overline{\kappa} $. By using the fact that $\left\lceil x\right\rceil < x+1$, we have
  \begin{align*}
    (2n_0+1)\overline{\kappa}<{\kappa},
  \end{align*}
  which implies that the desired bandwidth condition is satisfied.
  
  Now we let $U\Sigma V^\top$ be the SVD of $A$. That is, $\Sigma:=\diag(\sigma_1,\cdots,\sigma_{\rank(A)})$, $\sigma_{>0}(A)=\{\sigma_1,\cdots,\sigma_{\rank(A)}\}$, and $U,V$ are matrices with orthonormal columns. We now observe that:
  \begin{align}\label{eqn:pn-1}
    \begin{aligned}[t]
      P_{2n_0+1}(A)
      &\;\;\stackrel{\mathclap{\cref{eqn:pn}}}{=}\;\;\sum_{j=0}^{n_0} a_{2j+1} U \Sigma V^\top (V\Sigma^2 V^\top)^{j}\\ 
      &\;\;=\;\;U\left(\sum_{j=0}^{n_0} a_{2j+1}   \Sigma^{2j+1}\right) V^\top\\
      &\;\;\stackrel{\mathclap{\cref{eqn:pn}}}{=}\;\;U p_{2n_0+1}(\Sigma) V^\top,
    \end{aligned}
  \end{align}
  where the second equality follows from the orthonormality of the columns of $V$. 
  We now recall that the unique Moore-Penrose inverse can be expressed by $A^+=U\Sigma^{-1}V^\top$, and \WT{observe by following the same lines of reasoning as the $A \succeq 0$ case that}
  \begin{equation}
    \color{black}
    \|A^+-P_{2n+1}(A)\|_2 \stackrel{\mathclap{\cref{eqn:pn}}}{=}\; \max_{\sigma\in\sigma_{>0}(A)} | 1/\sigma - p_{2n_0+1}(\sigma)| \stackrel{\mathclap{\cref{eqn:polbound}}}{\leq} g(n_0).
  \end{equation}
  % \begin{align*}
  %   &\|A^+-P_{2n+1}(A)\|_2\\
  %   &\stackrel{\mathclap{\cref{eqn:pn-1}}}{=}\; \|U(\Sigma^{-1} - p_{2n_0+1}(\Sigma))V^\top  \|_2\\
  %   &= \;\|\Sigma^{-1} -p_{2n_0+1}(\Sigma)\|_2\\
  %   &= \left\|
  %     \begin{bmatrix}
  %       (1/\sigma_1)-\sum_{j=0}^{n_0} a_{2j+1}   \sigma^{2j+1}_1\\
  %       &\ddots\\
  %       &&(1/\sigma_{\rank(A)})-\sum_{j=0}^{n_0} a_{2j+1}   \sigma^{2j+1}_{\rank(A)}
  %     \end{bmatrix}
  %     \right\|_2\\
  %   &\stackrel{\mathclap{\cref{eqn:pn}}}{=}\; \max_{\sigma\in\sigma_{>0}(A)} | 1/\sigma - p_{2n_0+1}(\sigma)|\\
  %   &\stackrel{\mathclap{\cref{eqn:polbound}}}{\leq} g(n_0),
  % \end{align*}
  % where the second inequality follows from the orthonormality of the columns of $U$ and $V$, and the third inequality follows from the diagonality of $\Sigma^{-1}$ and $p_{2n_0+1}(\Sigma)$.
\end{proof}

\cref{lem:decay} suggests that if we have a polynomial approximation of $1/x$ and the associated error-bounding function, one can use that result to establish the error bound of the banded approximations of $A^+$. Now it remains to construct the odd (for the general system) or general (for positive semi-definite case) polynomial approximations of $1/x$ over $[a,b]$.
Note that deriving a naive approximation is not difficult, but ensuring that the approximation is optimal or sufficiently tight, at least in an asymptotic regime, is more challenging.
For the positive semi-definite case, one can apply the standard results in the approximation theory, which utilize Chebyshev polynomials, to obtain an optimal approximation.
For the general case, we use an odd polynomial obtained by modifying Chebyshev polynomials that is not necessarily optimal, but we will show in \Cref{sec:control} that this provides a sufficiently tight bound for analyzing time-discretized optimal control problems, where the mesh size tends to zero.
The next lemma constructs $\{p_{2n+1}(\cdot)\in P_{2n+1}\}_{n=0}^\infty$, $\{q_{n}(\cdot)\in P_{n}\}_{n=0}^\infty$, $g(\cdot)$, and $h(\cdot)$ such that \cref{eqn:polbound} and \cref{eqn:polbound-psd} hold.

\begin{lemma}\label{lem:polynomial}
\WT{Suppose that $A \in \mathbb{R}^{m \times n}$ is $(\overline{\kappa},V,d,\cI,\cJ)$-banded and $\sigma_{>0}(A)\subseteq [a,b]\subseteq \mathbb{R}_{>0}$. }Then there exists a sequence $\{p_{2n+1}(\cdot)\in P_{2n+1}\}_{n=0}^\infty$ of odd polynomials such that \cref{eqn:polbound} holds with
  \begin{align}\label{eqn:the-g}
    g(n) =\dfrac{4\left(b+a\right)^{3/2}}{ab(b-a)^{1/2}}\left(\frac{b-a}{b+a}\right)^{ n+1}
  \end{align}
  Furthermore, there exists a sequence $\{q_{n}(\cdot)\in P_{n}\}_{n=0}^\infty$ of polynomials such that \cref{eqn:polbound-psd} holds with
  \begin{align}\label{eqn:the-h}
    h(n) = \frac{(\sqrt{b}+\sqrt{a})^2}{2a b}\left(\frac{\sqrt{b}-\sqrt{a}}{\sqrt{b}+\sqrt{a}}\right)^{n+1}.
\end{align}
\end{lemma}
\begin{proof}
  Equation \cref{eqn:polbound-psd} with $\{q_n(\cdot)\}_{n=0}^\infty$ and $h(\cdot)$ in \cref{eqn:the-h} follows from \cite[proposition 2.1]{demko1984decay}. Thus, in what follows, we focus on proving \cref{eqn:polbound} with $\{p_{2n+1}(\cdot)\}_{n=0}^\infty$ and $g(\cdot)$ in \cref{eqn:the-g}.

  First, we review the definition and key properties of Chebyshev polynomials (see \cite{rivlin_introduction_1981} for more detailed introduction). The Chebyshev polynomial of the first kind $T_n(\cdot)$ is defined as the polynomial satisfying
  \begin{align}\label{eqn:ch-def}
    T_n (\cos\theta) = \cos(n\theta),
  \end{align}
  and can be computed using the following recursive formula:
  \begin{align}\label{eqn:pro-1}
    T_0(y)= 1;\quad
    T_1(y) = y;\quad
    T_{n+1}(y) = 2y T_n(y) - T_{n-1}(y),\quad n=1,2,\cdots.
  \end{align}

  We now fix $y\in[-1,1],\; t\in (0,1)$ and observe that the following equation holds:
  \begin{align}\label{eqn:pro-2}
    \sum_{j=0}^\infty t^j T_j(y) = \frac{1-ty}{1-2ty+t^2},
  \end{align}
  which can be confirmed by setting $\theta=\cos^{-1} y$ and:
  \begin{align*}
    \sum_{j=0}^\infty t^j T_j(\cos\theta)
  &= \Re\left[\sum_{j=0}^\infty (t e^{\sqrt{-1}\theta})^j\right]\\
  &= \Re\left[\frac{1}{1-te^{\sqrt{-1} \theta}}\right]\\
  &= \frac{1 - t\cos\theta }{1-2t\cos\theta +t^2},
\end{align*}
where the first equality follows from Euler's formula, the second equality is obtained by applying the summation of geometric series and the fact that $|te^{\sqrt{-1}\theta}| < 1$, and the third equality can be obtained through simplification.
We now set
\begin{align}\label{eqn:c-t}
  c := \frac{1+t^2}{2t},
\end{align}
and with the arithmetic-geometric mean argument, one can observe that $c>1$. We now can further simplify \cref{eqn:pro-2} using $c$ as follows:
\begin{align*}
  \sum_{j=0}^\infty t^j T_j(y)
  &\stackrel{\mathclap{\cref{eqn:c-t}}}{\;\;=\;\;} \frac{1-ty}{2t(c-y)} \\
  &\;\;=\;\; \frac{1}{2} + \frac{1-tc}{2t(c-y)}\\
  &\stackrel{\mathclap{\cref{eqn:c-t}}}{\;\;=\;\;} \frac{1}{2} + \frac{1-t^2}{4t(c-y)}.
\end{align*}
By multiplying both sides by $4t/(1-t^2)$ and rearranging, we obtain the following:
\begin{align}\label{eqn:chebyshev}
  \frac{1}{c-y} &= -\frac{2t}{1-t^2} + \frac{4t}{1-t^2} \sum_{j=0}^\infty t^j T_j(y).
\end{align}
We now claim that:
\begin{align}\label{eqn:claim}
  \sum_{j=n+1}^\infty t^j T_j(y)&= t^{n}\frac{T_{n+1}(y)-t T_{n}(y)}{2(c-y)},\quad n=0,1,\cdots.
\end{align}
We prove this by induction. If $n=0$, 
\begin{align*}
  \sum_{j=1}^\infty t^j T_j(y)- \frac{T_1(y)-tT_0(y)}{2(c-y)}
  &\stackrel{\mathclap{\cref{eqn:pro-1}}}{\;\;=\;\;}\sum_{j=0}^\infty t^j T_j(y) - 1 - \frac{y-t }{2(c-y)}\\
  &\stackrel{\mathclap{\cref{eqn:chebyshev}}}{\;\;=\;\;}\frac{1}{2}+\frac{1-t^2}{4t(c-y)}-1 - \frac{y-t}{2(c-y)}\\
  &\;\;=\;\; \frac{1-t^2 - 2t(c-y) - 2t(y-t)}{4t(c-y)}\\
  &\;\;=\;\; \frac{1+t^2-2tc}{4t(c-y)}\\
  &\stackrel{\mathclap{\cref{eqn:c-t}}}{\;\;=\;\;}0.
\end{align*}
We now assume that the claim holds for $n-1$ and observe that
\begin{align*}
  \sum_{j=n+1}^\infty t^j T_j(y)
  &= \sum_{j=n}^\infty t^j T_j(y) - t^nT_n(y)\\
  &\stackrel{\mathclap{\cref{eqn:claim}}}{=} t^{n-1}\frac{T_{n}(y)-t T_{n-1}(y)}{2(c-y)} - t^n T_n(y)\\
  &\stackrel{\mathclap{\cref{eqn:pro-1}}}{=}t^{n-1}\frac{T_{n}(y)-t (2yT_n(y) - T_{n+1}(y)) - 2t(c-y)T_n(y)}{2(c-y)}  \\
  &=t^{n-1}\frac{ tT_{n+1}(y) - (2ct -1)T_n(y)}{2(c-y)}\\
  &\stackrel{\mathclap{\cref{eqn:c-t}}}{=}t^{n}\frac{ T_{n+1}(y) - tT_n(y)}{2(c-y)}.
\end{align*}
This completes the induction.
By multiplying \cref{eqn:chebyshev} by $\sqrt{c-y}$ and rearranging, we obtain
\begin{align}\label{eqn:residual-2}
  \frac{1}{\sqrt{c-y}} + \frac{2t\sqrt{c-y}}{1-t^2}- \frac{4t\sqrt{c-y}}{1-t^2} \sum_{j=0}^n t^j T_j(y)= \frac{4t\sqrt{c-y}}{1-t^2} \sum_{j=n+1}^\infty t^j T_j(y).
\end{align}
By \cref{eqn:claim}, we have
\begin{subequations}\label{eqn:jokar}
  \begin{align}
    \sqrt{c-y}\sum_{j=n+1}^\infty t^j T_j(y)
    % &= (\sqrt{c-y}) t^{n}\frac{T_{n+1}(y)-t T_{n}(y)}{2(c-y)}\\
    &= t^n\frac{T_{n+1}(y)-t T_{n}(y)}{\sqrt{c-y}}\\
    &= t^n\frac{T_{n+1}(y)- T_{n}(y)}{\sqrt{c-y}} + t^n\frac{(1-t)T_{n}(y)}{\sqrt{c-y}}.
      \label{eqn:jokar-last}
  \end{align}
\end{subequations}
The first term in \cref{eqn:jokar-last} satisfies: (i) when $y=1=\cos 0$, we observe that
\begin{subequations}\label{eqn:sqrt2-1}
\begin{align}
  \begin{aligned}[t]
  \left|\frac{T_{n+1}(y)- T_{n}(y)}{\sqrt{c-y}}\right|
  &\;\;=\;\; \left|\frac{T_{n+1}(\cos 0)- T_{n}(\cos 0)}{\sqrt{c-1}}\right|\\
  &\stackrel{\mathclap{\cref{eqn:ch-def}}}{\;\;=\;\;} 0
  \end{aligned}
\end{align}
(ii) when $y=-1=\cos(\pi)$, we observe that
\begin{align}
  \begin{aligned}[t]
  \left|\frac{T_{n+1}(y)- T_{n}(y)}{\sqrt{c-y}}\right|
  &\;\;=\;\; \left|\frac{T_{n+1}(\cos \pi)- T_{n}(\cos \pi)}{\sqrt{c+1}}\right|\\
  &\stackrel{\mathclap{\cref{eqn:ch-def},\; c>1}}{\;\;\leq\;\;} \sqrt{2}.
  \end{aligned}
\end{align}
(iii) when $y\in (-1,1)$, we let $\theta=\cos^{-1}y$ and observe that
\begin{align}
  \begin{aligned}[t]
  \left|\frac{T_{n+1}(y)- T_{n}(y)}{\sqrt{c-y}}\right|
  % &\leq \left|\frac{T_{n+1}(y)- T_{n}(y)}{\sqrt{1-y}}\right|\left|\frac{\sqrt{1-y}}{\sqrt{c-y}}\right|\\\nonumber
  &\stackrel{\mathclap{c>1}}{\leq} \left|\frac{\cos((n+1)\theta)- \cos(n\theta)}{\sqrt{1-\cos(\theta)}}\right|\\
  &\leq \left|\frac{2\sin((n+1/2)\theta)\sin(\theta/2)}{\sqrt{2\sin^2(\theta/2)}}\right|\\
  &\leq\sqrt{2},
  \end{aligned}
\end{align}
\end{subequations}
where the second inequality follows from trigonometric formulas, and the last inequality can be obtained by simplification.

For the second term in \cref{eqn:jokar-last}, we have
\begin{align}\label{eqn:sqrt2-2}
  \left|\frac{(1-t)T_{n}(y)}{\sqrt{c-y}}\right|
  &\leq \frac{1-t}{\sqrt{c-1}}\\\nonumber
  &\stackrel{\mathclap{\cref{eqn:c-t}}}{\leq} \sqrt{2t}\\\nonumber
  &\leq \sqrt{2},
\end{align}
where the first inequality follows from the fact that $y\in[-1,1]$ and $|T_{n}(y)|\leq 1$ for $y\in[-1,1]$ (by the definition in \cref{eqn:ch-def}), and the last inequality follows from $t\in(0,1)$.

Compiling \cref{eqn:residual-2}, \cref{eqn:jokar}, \cref{eqn:sqrt2-1}, and \cref{eqn:sqrt2-2} yields
\begin{align}\label{eqn:residual}
  \left|\frac{1}{\sqrt{c-y}} + \frac{2t}{1-t^2}\sqrt{c-y}- \frac{4t}{1-t^2} \sqrt{c-y}\sum_{j=0}^n t^j T_j(y)\right|&\leq \frac{8\sqrt{2}t^{n+1}}{1-t^2}.
\end{align}
Now, we fix $0<a<b$ and $x\in[a,b]$, and let
\begin{align}\label{eqn:yc}
  y=-\frac{2x^2}{b^2-a^2} + \frac{b^2+a^2}{b^2-a^2},\quad t = \frac{b-a}{b+a}.
\end{align}
One can confirm that $y\in[-1,1]$ and $t\in(0,1)$. Also, this implies
\begin{align*}
  c=\frac{b^2+a^2}{b^2-a^2},\quad \sqrt{c-y} = \left(\frac{2}{b^2-a^2}\right)^{1/2} x.
\end{align*}
Then, we substitute $y$ and $c$ to \cref{eqn:residual} and multiply both sides by $(2/(b^2-a^2))^{1/2}$ to yield 
\begin{align*}
  &\left|\frac{1}{x} - \frac{2}{b^2-a^2}\left(- \frac{2t}{1-t^2} x +\frac{4t}{1-t^2}x\sum_{j=0}^n t^j T_j\left(-\frac{2x^2}{b^2-a^2} + \frac{b^2+a^2}{b^2-a^2}\right)\right)\right|\\
  &\;\;\leq\;\; \frac{16t^{n+1}}{(b^2-a^2)^{1/2}(1-t^2)}\\
  &\stackrel{\mathclap{\cref{eqn:yc}}}{\;\;\leq\;\;} \frac{4\left(b+a\right)^{3/2}}{ab(b-a)^{1/2}}\left(\frac{b-a}{b+a}\right)^{n+1}.
\end{align*}
We observe that the right-hand side is $g(n)$.
Finally, we set
\begin{align*}
  p_{2n+1}(x) &:=\frac{2}{b^2-a^2}\left(-\frac{2t}{1-t^2}x + \frac{4t}{1-t^2} x\sum_{j=0}^\infty t^j T_j\left(-\frac{2x^2}{b^2-a^2}+\frac{b^2+a^2}{b^2-a^2}\right)\right)
\end{align*}
and confirm that the degree of $p_{2n+1}(\cdot)$ is not greater than $2n+1$ and $p_{2n+1}(\cdot)$ is an odd polynomial. Therefore, \cref{eqn:polbound} holds with $\{p_{2n+1}(\cdot)\}_{n=0}^\infty$ and $g(\cdot)$.
\end{proof}

\cref{lem:polynomial} establishes the desired polynomial approximation results. The proof for the odd polynomial in \cref{lem:polynomial} is inspired by the proof for the general case, which is provided in \cite{jokar_best_2005}.

By combining \cref{lem:polynomial} with \cref{lem:decay}, one can obtain the exponential bound for the banded approximations established in \cref{thm:main}. We write the proof below for completeness.

\begin{proof}[Proof of \cref{thm:main}]
  Follows from \cref{lem:decay,lem:polynomial}.
\end{proof}

Now we derive an important corollary of \cref{thm:main}: the exponential decay of the off-diagonal elements of $A^+$.
Specifically, with the results in \cref{thm:main}, one can straightforwardly show that the off-diagonal block
$$A^+\left[\bigcup_{i\in V_1} I_i,\bigcup_{i\in V_2} J_i\right]$$
for $V_1,V_2\subseteq V$ decays exponentially with respect to the distance $d(V_1,V_2)$ between $V_1$ and $V_2$ on the metric space $(V,d)$, where
$$d(V_1,V_2) := \min_{i \in V_1, j \in V_2} d(i,j),\quad \forall V_1,V_2\subseteq V.$$
The following corollary establishes such an exponentially decaying bound for the off-diagonal blocks.

\begin{corollary}\label{cor:main}
  Under the assumptions of \cref{thm:main} with $\cI:=\{I_i\}_{i\in V}$ and $\cJ:=\{J_i\}_{i\in V}$, the following holds for all $V_1,V_2\subseteq V$:
  \begin{align*}
    \left\|A^+\left[\bigcup_{i\in V_1} I_i,\bigcup_{i\in V_2} J_i\right]\right\|_2 &\leq f_A\left(\frac{d(V_1,V_2)-1}{\overline\kappa}, a, b\right),
  \end{align*}
  where $f_A(\cdot)$ is defined in \cref{eqn:f}.
\end{corollary}
\begin{proof}
  We choose $\kappa:=d(V_1,V_2)-1$ and let $\widetilde{A}$ be a $(d(V_1,V_2)-1)$-banded approximation satisfying \cref{eqn:main} (exists due to \cref{thm:main}). This yields
\begin{align}\label{eqn:off-zero}
  \widetilde{A}\left[\bigcup_{i\in V_1} I_i,\bigcup_{i\in V_2} J_i\right]=0.
\end{align}
The claim follows from
\begin{align*}
  \left\|A^+\left[\bigcup_{i\in V_1} I_i,\bigcup_{i\in V_2} J_i\right]\right\|_2 
  &\stackrel{\mathclap{\cref{eqn:off-zero}}}{\;\;=\;\;}
    \left\|(A^+-\widetilde{A})\left[\bigcup_{i\in V_1} I_i,\bigcup_{i\in V_2} J_i\right]\right\|_2 \\
  &\;\;\leq\;\;\left\|A^+ - \widetilde{A}\right\|_2\\
  &\;\;\leq\;\;  f_A\left(\frac{d(V_1,V_2)-1}{\overline\kappa},a,b\right),
\end{align*}
where the first inequality follows from that the norm of a submatrix is not greater than the original matrix, and the last inequality follows from the construction of $\widetilde{A}$.
\end{proof}

\subsection{Comparison with Existing Results}

The exponentially decaying structure in the inverse of banded systems has been studied in the literature \cite{demko1984decay,bickel2012approximating,shin2022exponential}.
We now compare our results with the existing bounds in the literature.
We focus on the general (indefinite) case, as our main contribution lies in the bound for the general cases.
We wrap the proof of the existing results based on the framework of \cref{lem:decay} to ensure consistency and fair comparison.
Two propositions in this section establish the bound discussed in \Cref{sec:intro}.

We first discuss the result of \cite{demko1984decay} in the following proposition.
% \frac{4\left(b+a\right)^{3/2}}{ab(b-a)^{1/2}}
\begin{proposition}\label{prop:demko}
  Under the assumptions of \cref{thm:main}, there exists a $\kappa$-banded approximation $\widetilde{A}\in\mathbb{R}^{n\times m}$ such that:
  \begin{align*}
    \left\|A^+- \widetilde{A} \right\|_2\leq  \dfrac{(a+b)^2}{2a^2b} \left(\dfrac{b-a}{b+a}\right)^{\lceil\frac{\kappa/\overline{\kappa}-1}{2}\rceil}.
  \end{align*}
\end{proposition}
\begin{proof}
  From \cite[proposition 2.1]{demko1984decay}, there exists $\{q_n\in P_n\}_{n=0}^\infty$ such that
\begin{align}\label{eqn:demko}
  |1/x^2 - q_n(x^2) | \leq \frac{(a+b)^2}{2a^2b^2} \left(\frac{b-a}{b+a}\right)^{n+1},\quad \forall x\in[a,b].
\end{align}
By triangle inequality,
\begin{align}\label{eqn:demko-poly}
  |1/x - xq_n(x^2) |
  &\;\;\leq\;\; |x||1/x^2 - q_n(x^2) |\\\nonumber
  &\stackrel{\mathclap{\cref{eqn:demko}}}{\;\;\leq\;\;} \frac{(a+b)^2}{2a^2 b} \left(\frac{b-a}{b+a}\right)^{n+1} .
\end{align}
Here, $xq_n(x^2)$ is a degree-$(2n+1)$ odd polynomial and satisfies \cref{eqn:polbound}. From \cref{lem:decay} and \cref{eqn:demko-poly}, we obtain the desired result.
\end{proof}

Note that in the original paper \cite{demko1984decay}, the authors provide a bound for the off-diagonal entries of $A^{-1}$ for positive definite $A$. The result is more general in the sense that it can handle infinite-dimensional linear systems. However, the paper primarily focuses on positive definite systems, and only a loose bound is obtained for general systems by using $A^{-1} = A^*(AA^*)^{-1}$. For example, the bound in \cite[proposition 2.3]{demko1984decay} is worse than the bound in \cref{prop:demko} in that it contains additional factors regarding $\overline{\kappa}$. Several ideas employed in this work, e.g., a generalized notion of bandedness and the analysis of Moore-Penrose inverses, are discussed in \cite[propositions 4.5 and 5.1]{demko1984decay}.

We now discuss the bound established in \cite{shin2022exponential,bickel2012approximating}. We note that the same polynomial approximation has been used in both papers.
\begin{proposition}\label{prop:shin}
  Under the assumptions of \cref{thm:main}, there exists a $\kappa$-banded approximation $\widetilde{A}\in\mathbb{R}^{n\times m}$ such that:
  \begin{align*}
    \left\|A^+- \widetilde{A} \right\|_2\leq  \dfrac{b}{a^2} \left(\dfrac{b^2-a^2}{b^2+a^2}\right)^{\lceil\frac{\kappa/\overline{\kappa}-1}{2}\rceil}.
  \end{align*}
\end{proposition}
\begin{proof}
One can easily observe that for $x\in [a,b]$,
\begin{align*}
  1/x
  &= \frac{2x}{b^2+a^2}\frac{1}{1-(1-\frac{2x^2}{b^2+a^2})} \\
  &= \sum_{j=0}^\infty \frac{2x}{b^2+a^2}\left(1-\frac{2x^2}{b^2+a^2}\right)^j,
\end{align*}
where the second equality follows from the observation that $|1-2x^2/(b^2+a^2)|<1$ for $x\in[a,b]$. 
By rearranging, we obtain
\begin{align*}
  \frac{1}{x} - \sum_{j=0}^n \frac{2x}{b^2+a^2}\left(1-\frac{2x^2}{b^2+a^2}\right)^j
  &= \sum_{j=n+1}^\infty\frac{2x}{b^2+a^2}\left(1-\frac{2x^2}{b^2+a^2}\right)^{j}\\
  &\leq \frac{2b}{b^2+a^2}\left(\frac{b^2-a^2}{b^2+a^2}\right)^{n+1}\frac{b^2+a^2}{2a^2}\\
  &\leq \frac{b}{a^2}\left(\frac{b^2-a^2}{b^2+a^2}\right)^{n+1}.
    % \tag*{\qedhere}
\end{align*}

\end{proof}
The proof is simple in the sense that it does not involve the use of Chebyshev polynomial. However, the bound is worse than the bound in \cref{prop:demko} in that its decay rate is slower. In \cite{shin2022exponential}, the inverse of graph-induced banded systems has been analyzed. This is more general than \cite{demko1984decay}, but is still more specific than the bandedness considered in this paper, which is based on a more general metric space. In \cite[section 4.1]{bickel2012approximating}, the generalization of the results to systems induced by metric spaces is discussed. 

Now, one can observe that the bound in \cref{thm:main} is tighter than the bounds in \cref{prop:demko,prop:shin} \WT{when $b$ is sufficiently large.}
First, the bound in \cref{thm:main} has a faster decay rate than \cref{prop:shin} and the same decay rate as \cref{prop:demko}.
Second, the constant factor in \cref{thm:main} is smaller than the constant factor in \cref{prop:demko,prop:shin}.
In particular, both the constant factors in \cref{prop:demko,prop:shin} are $O(b/a^2)$, while the constant factor in \cref{thm:main} is $O(1/a)$.
This is important, as this tighter bound allows for the decay bound to be uniform to the discretization mesh size.
In \Cref{sec:control}, we will see that this improvement enables the analysis of the discretization of infinite-dimensional linear systems.

\section{Singular Value Bounds of Saddle Point Systems}\label{sec:kkt}
The exponential perturbation bounds in \cref{thm:main} are expressed in terms of the singular value bounds of $A$. Thus, investigating the bounds of the singular values is of interest. In this section, we consider a special case where $A$ is obtained as the optimality condition of an equality-constrained quadratic program and derive their singular value bounds. This system is often called a saddle point system.

In particular, we consider the following matrix:
\begin{align}\label{eqn:kkt}
  A:=
  \begin{bmatrix}
    G&F^\top\\
    F & 0
  \end{bmatrix};\quad
  \text{$G$ is symmetric.}
\end{align}
Note that systems involving the matrix in \cref{eqn:kkt} arise from linear equality-constrained quadratic optimization problems of the following form: 
\begin{align}\label{eqn:opt} 
  \min_{x} x^\top(Gx-g)\; \st\; Fx = f.
\end{align}
Such a problem is often considered within nonlinear optimization algorithms to compute Newton's step direction \cite{nocedal_numerical_2006} or within the sensitivity analysis of nonlinear programming solution mappings \cite{robinson_strongly_1980}.

The following theorem establishes the bound of singular values of the system in \cref{eqn:kkt}.
\begin{theorem}\label{thm:uniform}
  Suppose that the following holds for $\theta_1\geq 0$, and $\theta_2,\theta_3,\theta_4>0$:
  \begin{align}\label{eqn:uniform-ass}
     G\preceq \theta_1 I ,\quad \theta_2 I \preceq FF^\top \preceq \theta_3 I,\quad G + F^\top F \succeq \theta_4 I.
  \end{align}
  Then, the singular values of $A$ satisfy the following:
  \begin{align*}
    \sigma(A)\in \left[\left( \left(\frac{1+\theta_1/\theta_2}{\theta_4}\right)^{1/2}  + \max\left(\frac{1}{\theta_4},\frac{\theta_1}{\theta_2}\right)  \right)^{-1}, \theta_1+\theta_3^{1/2}\right].
  \end{align*}
\end{theorem}
\begin{proof}
  The upper bound can be obtained by
  \begin{align*}
    \left\|
    \begin{bmatrix}
      G & F^\top\\
      F
    \end{bmatrix}
    \right\|_2
    & \leq 
      \left\|
      \begin{bmatrix}
        G & 0\\
        0 &\\
      \end{bmatrix}
      \right\|_2
      +
      \left\|
      \begin{bmatrix}
        0& F^\top\\
        F&\\
      \end{bmatrix}
      \right\|_2
    \\
        &\leq \|G\|_2 + \|F\|_2\\
        &\stackrel{\mathclap{\cref{eqn:uniform-ass}}}{\leq} \theta_1 + \theta_3^{1/2}.
  \end{align*}
  
  We now focus on the lower bound. First, we observe that:
  \begin{align}\label{eqn:inv}
    A^{-1}
    &=
    \begin{bmatrix}
      T - TF^\top (FTF^\top)^{-1} FT & TF^\top (FTF^\top)^{-1}\\
      (FTF^\top)^{-1} FT &  I-(FTF^\top)^{-1}\\
    \end{bmatrix}
  \end{align}
  where $T:=(G+F^\top F)^{-1}$. This can be simply verified by left-multiplying $A$ to \cref{eqn:inv}. Now, let $USV^\top$ be the SVD of $FT^{1/2}$. Then, we have that $$(FTF^\top)^{-1} = (USV^\top V SU^\top)^{-1} = US^{-2}U^\top,$$
  where the second equality follows from that $F$ has full row-rank due to \cref{eqn:uniform-ass}. Furthermore, from \cref{eqn:inv}, we have:
  \begin{align}\label{eqn:Ainv}
    A^{-1}
    &=
    \begin{bmatrix}
      T - T^{1/2}VSU^\top  (US^{-2}U^\top) USV^\top T^{1/2} & T^{1/2}VSU^\top  (US^{-2}U^\top)\\
      (US^{-2}U^\top) USV^\top T^{1/2} &  I-(US^{-2}U^\top)\\
    \end{bmatrix}\\\nonumber
    &=
      \begin{bmatrix}
        T^{1/2}(I-VV^\top) T^{1/2} & T^{1/2}VS^{-1}U^\top\\
        US^{-1}V^\top T^{1/2} &  U(I-S^{-2})U^\top
      \end{bmatrix}
  \end{align}
  We now observe from \cref{eqn:uniform-ass} that
  \begin{align}\label{eqn:T}
    T &\preceq (1/\theta_4){I}.
  \end{align}
  Also, we observe
  \begin{align}\label{eqn:S}
    S^{-2}
    &\quad=\quad U^\top (FTF^\top)^{-1}U\\\nonumber
    &\quad=\quad U^\top (F(G + F^\top F)^{-1}F^\top)^{-1}U\\\nonumber
    &\stackrel{\mathclap{\cref{eqn:uniform-ass}}}{\quad\preceq\quad} U^\top\left(F\left(\theta_1 {I} + F^\top F\right)^{-1}F^\top\right)^{-1} U\\\nonumber
    &\quad\preceq\quad  U\left( \frac{1}{\theta_1} F\left( {I} + \frac{1}{\theta_1}F^\top F\right)^{-1}F^\top\right)^{-1}U^\top\\\nonumber
    &\stackrel{\mathclap{\text{\cite[fact 2.13.18]{bernstein2009matrix}}}}{\quad\preceq\quad}  U\left( {I} -\left({I}+ \frac{1}{\theta_1}FF^\top\right)^{-1} \right)^{-1}U^\top\\\nonumber
    &\stackrel{\mathclap{\cref{eqn:uniform-ass}}}{\quad\preceq\quad}  U\left( {I} -\left({I}+ \frac{\theta_2}{\theta_1}{I}\right)^{-1} \right)^{-1}U^\top\\\nonumber
    &\quad\preceq\quad (1+\theta_1/\theta_2) {I}.
  \end{align}
  Therefore,
  \begin{align*}
    \|A^{-1}\|_2
    &
      {
      \footnotesize
      \stackrel{\mathclap{\cref{eqn:Ainv}}}{\;\;\;\leq\;\;\;} \left\|
      \begin{bmatrix}
        & T^{1/2}VS^{-1}U^\top\\
        US^{-1}V^\top T^{1/2} 
      \end{bmatrix}
      \right\|_2 +
      \left\|
      \begin{bmatrix}
        T^{1/2}({I}-VV^\top) T^{1/2} & \\
                                     &  U({I}-S^{-2})U^\top
      \end{bmatrix}
      \right\|_2
      }\\
    & \;\;\;\leq\;\;\; \|US^{-1}V^\top T^{1/2} \|_2 + \max\left(\|T^{1/2}({I}-VV^\top) T^{1/2} \|_2, \|U({I}-S^{-2})U^\top\|_2\right)\\
    & \;\;\;\leq\;\;\; \|S^{-1}\|_2\|T^{1/2}\|_2 + \max(\|T\|_2 , \|S^{-2}-{I}\|_2)
    \\
    & \stackrel{\mathclap{\cref{eqn:T},\cref{eqn:S}}}{\;\;\;\leq\;\;\;} \left(\frac{1+\theta_1/\theta_2}{\theta_4}\right)^{1/2}  + \max\left(\frac{1}{\theta_4},\frac{\theta_1}{\theta_2}\right).
      % \tag*{\qedhere}
  \end{align*}
\end{proof}

The conditions in \cref{eqn:uniform-ass} are related to the regularity conditions for the optimization problem in \cref{eqn:opt}. These conditions assume the boundedness of $G$ and $F$, the linear independence constraint qualification (which assumes a full low-rank constraint Jacobian), and the second-order sufficiency condition (which assumes a positive definite reduced Hessian), all satisfied with explicitly introduced bounds. Specifically, the operator norm of $G$ and $F$, the smallest singular value of $F$, and the smallest eigenvalue of the reduced Hessian are bounded above or below. A necessary condition for the third condition---reduced Hessian being positive definite---is $G+\delta F^\top F$ being positive definite for sufficiently large $\delta$ \cite[lemma 16.1]{nocedal_numerical_2006}. We are assuming $\delta=1$ without loss of \WT{generality in} \cref{thm:uniform}, as the theorem can be applied after rescaling $G$ and $F$.

\cref{thm:uniform} improves the bound established in our previous work \cite{shin2022exponential} in that the lower bound in \cref{thm:uniform} does not directly depend on the upper bound of $\|A\|$; rather, it is only dependent on $\|G\|$. This is advantageous when analyzing the discretization of infinite-dimensional linear systems, as the upper bound of $F$ often diverges as the mesh refines. \WT{This divergence is due to the fact that if the entries of $F$ involve finite difference approximation of derivatives, those entries diverge as the mesh refines.} The efficacy of \cref{thm:uniform} will be demonstrated in the next section.

\section{Application to Optimal Control}\label{sec:control}
Optimal control problems, either discrete-time or discretized continuous-time, exhibit banded structure when expressed as saddle point systems due to their dynamic constraints \cite{na_convergence_2022}.  Also, the optimal control problems over network systems often embed the network structure and exhibit banded structure when they are indexed by metric spaces over spatial domains \cite{shin_near-optimal_2023}. As such, strategies to exploit such exponential decay behaviors have been a long-standing topic in the control literature \cite{shin_near-optimal_2023,motee_optimal_2008,lin_perturbation-based_2021}.

In this section, we demonstrate the relevance of \cref{thm:main} to optimal control problems, in particular, for classical continuous-domain linear-quadratic optimal control problems. The exponential bound established in \cref{thm:main} and the bound of the singular values of saddle point systems established in \cref{thm:uniform} allow us to show that the solution mapping exhibits the exponentially decaying structure in a uniform sense. Specifically, we will see that the bound established in \cref{thm:main} is uniform in the discretization mesh size, and thus, we can establish the continuous-domain exponential decay results using the discrete-domain result. 

\subsection{Settings}
We first introduce notation for function spaces.
\begin{align*}
  C([a,b]; \mathbb{R}^n)
  &:=\left\{f:[a,b]\rightarrow \mathbb{R}^{n}\;\middle|\; f(\cdot)\text{ is continuous in } (\mathbb{R}^{n},\|\cdot\|_2)\right\}\\
  L([a,b]; \mathbb{R}^n)
  &:= \left\{f:[a,b]\rightarrow \mathbb{R}^{n} \;\middle|\;  \langle z, f(\cdot)\rangle \text{ is measurable for all } z\in\mathbb{R}^{n}\right\}\\
  L^2([a,b]; \mathbb{R}^n)
  &:= \left\{f(\cdot)\in L([a,b];\mathbb{R}^{n})\;\middle|\;\|f(\cdot)\|_2:=\left(\int_{a}^b\|f(t)\|^2_2dt \right)^{1/2} < \infty\right\}
\end{align*}
One can easily see that the above spaces form vector spaces.
\WT{Additionally, $L^2([a,b]);\mathbb{R}^{n})$ is known to be a Hilbert space equipped with the inner product:}
\begin{equation}\label{eqn:norm}
  \color{black}
\langle f,g \rangle_{L^2([a,b]);\mathbb{R}^{n})} := \int_a^b f(t)^\top g(t)dt 
\end{equation}
We refer the readers to \cite[appendix A]{curtain2012introduction} for more details and properties of these function spaces.

We study the following linear-quadratic optimal control problem:
\begin{subequations}\label{eqn:ocp}
  \begin{align}
    \min_{s(\cdot),u(\cdot)}\;
    & \int_{0}^T
      \begin{bmatrix}
        s(t)\\u(t)
      \end{bmatrix}^\top
      \!\left(
      \frac{1}{2}
      \begin{bmatrix}
        C^\top C & 0\\
        0 & I
      \end{bmatrix}
      \begin{bmatrix}
        s(t)\\u(t)
      \end{bmatrix}
      -
      \begin{bmatrix}
        q(t)\\r(t)
      \end{bmatrix}
      \right) dt
      - \overline{\lambda}^\top s(T) \label{eqn:ocp_a}\\[4pt]
    \text{s.t.}\quad
    & s(0) = \overline{s}, \label{eqn:ocp_b}\\
    & \WT{\dot{s}(t)} = \Lambda s(t) + B u(t) + d(t), \quad t \in [0,T]. \label{eqn:ocp_c}
  \end{align}
\end{subequations}
where $s:[0,T]\rightarrow\mathbb{R}^{n_s}$ is the state variable, $u:[0,T]\rightarrow\mathbb{R}^{n_u}$ is the control variable, $C\in \mathbb{R}^{n_y\times n_s}$ is the state-output mapping, $\Lambda\in \mathbb{R}^{n_s\times n_s}$ is the dynamic mapping, $B\in \mathbb{R}^{n_s\times n_u}$ is the control-state mapping, $\overline{s},\overline{\lambda}\in\mathbb{R}^{n_s}$, $q:[0,T]\rightarrow\mathbb{R}^{n_s}$, $r:[0,T]\rightarrow\mathbb{R}^{n_u}$, and $d:[0,T]\rightarrow\mathbb{R}^{n_s}$ are data.

\WT{The stationarity condition of \cref{eqn:ocp}, which is often referred to as the Euler-Lagrange equations}, can be stated as follows:
\begin{align}\label{eqn:pont}
  \underbrace{
  \begin{bmatrix}
    C^\top C&&-\frac{d}{dt} - \Lambda^\top \\
            &I&-B^\top\\
    \frac{d}{dt} - \Lambda& -B\\
       E_0 && \\
              && E_T
  \end{bmatrix}
  }_{:= H}
  \begin{bmatrix}
    s(\cdot)\\
    u(\cdot)\\
    \lambda(\cdot)
  \end{bmatrix}=
  \begin{bmatrix}
    q(\cdot)\\
    r(\cdot)\\
    d(\cdot)\\
        \overline{s}\\
    \overline{\lambda}
  \end{bmatrix},
\end{align}
where $E_0s(\cdot) = s(0)$ and $E_T\lambda(\cdot) = \lambda(T)$. \WT{Since this is a linear quadratic problem with unconstrained controls, the Pontryagin's minimum principle guarantees that \cref{eqn:pont} is both necessary and sufficient for optimality.} We will denote $n_z = n_s+n_u+n_s$, $z(\cdot) = (s(\cdot), u(\cdot), \lambda(\cdot))$, and $p(\cdot)=(q(\cdot), r(\cdot), d(\cdot))$.

We seek solutions $z^\star(\cdot)=H^{-1}(p(\cdot), \overline{s}, \overline{\lambda})$ in the following mild sense:
\begin{subequations}\label{eqn:mild}
  \begin{align}
    \lambda^\star(t)
    &=  \int_{t}^T e^{\Lambda^\top (\tau-t)} (C^\top Cs^\star (\tau) - q(\tau)) d\tau + e^{\Lambda^\top(T-t)}  \overline{\lambda},\quad t\in[0,T]\\
    0
    &= r(t) -u^\star(t)+ B^\top \lambda^\star(t),\quad t\in[0,T]\\
    s^\star (t)
    &= e^{\Lambda t} \overline{s}  + \int_{0}^t e^{\Lambda (t-\tau)} (Bu^\star(\tau)+d(\tau)) d\tau,\quad t\in[0,T].
  \end{align}
\end{subequations}
\WT{The operator $H^{-1}$ is a map between the Hilbert spaces $Y := L^2([0,T];\mathbb{R}^{n_z})\times \mathbb{R}^{n_s}\times \mathbb{R}^{n_s}$ equipped with the product norm 
\begin{equation}
\lVert (p(\cdot), \overline{s}, \overline{\lambda}) \rVert_Y^2 := \lVert p(\cdot) \rVert_{L^2([0,T];\mathbb{R}^{n_z})}^2 + \lVert \overline{s} \rVert_2^2 + \lVert \overline{\lambda} \rVert_2^2
\end{equation}
and $X := L^2([0,T];\mathbb{R}^{n_z})$ equipped with the inner product in  \cref{eqn:norm}.}

We now define the key concepts used in our assumptions. These concepts are standard notions of regularity used in control theory literature \cite{anderson_optimal_2007}.
\begin{definition}
  Given $\Phi,\Lambda\in \mathbb{R}^{n_s\times n_s}$, $B\in \mathbb{R}^{n_s\times n_u}$, $C\in \mathbb{R}^{n_y\times n_s}$, $L\geq 1$, and $\alpha>0$, we define the following.
  \begin{enumerate}[label=(\alph*)]
  \item $\Phi$ is $(L,\alpha)$-stable if $\|e^{\Phi t}\|\leq Le^{-\alpha t}$ for all $t\geq 0$.
  \item $(\Lambda,B)$ is $(L,\alpha)$-stabilizable if there exists $K\in \mathbb{R}^{n_u\times n_s}$ such that $\|K\|\leq L$, and $\Lambda-BK$ is $(L,\alpha)$-stable.
  \item $(\Lambda,C)$ is $(L,\alpha)$-detectable if there exists $K\in \mathbb{R}^{n_s\times n_y}$ such that $\|K\|\leq L$, and $\Lambda-KC$ is $(L,\alpha)$-stable.
  \end{enumerate}
\end{definition}

We now assume that the control system considered in \cref{eqn:ocp} is bounded, stabilizable, and detectable, with explicitly given bounds and convergence rates.
\begin{assumption}\label{ass:main}
  We assume the following conditions for given $\Lambda\in \mathbb{R}^{n_s\times n_s}$, $B\in \mathbb{R}^{n_s\times n_u}$, $C\in \mathbb{R}^{n_y\times n_s}$, $L\geq 1$, and $\alpha>0$.
  \begin{enumerate}[label=(\alph*)]
  \item $\|\Lambda\|,\|B\|,\|C\|,\leq L$.
  \item $(\Lambda,B)$ is $(L,\alpha)$-stabilizable.
  \item $(\Lambda,C)$ is $(L,\alpha)$-detectable.
  \end{enumerate}
\end{assumption}
\WT{
  Equipped with \Cref{ass:main}, the  existence of solution mapping for \cref{eqn:pont} is ensured due to \cite[theorem 10]{grune2020exponential}. Below, we adapt the original result to our problem in \cref{eqn:pont}.
\begin{theorem}\label{thm:bdd-sol-exists}
  Under \Cref{ass:main}, there exists a bounded linear map $H^{-1}: Y \to X$ such that for all $p(\cdot)\in L^2([0,T];\mathbb{R}^{n_z})$ and $\overline s, \overline\lambda\in\mathbb{R}^{n_s}$,
  \begin{align}\label{eqn:zstar}
    z^\star(\cdot):=H^{-1} (p(\cdot),\overline{s},\overline{\lambda})
  \end{align}
  satisfies \cref{eqn:mild}.
\end{theorem}
\begin{proof}
  Rewrite \cref{eqn:pont} as 
  \begin{align}\label{eqn:pont-reduced}
    \underbrace{
    \begin{bmatrix}
      C^\top C&-\frac{d}{dt} - \Lambda^\top \\
      \frac{d}{dt} - \Lambda& -BB^\top\\
      E_0 & \\
              & E_T
    \end{bmatrix}
    }_{:= M}
    \begin{bmatrix}
      s(\cdot)\\
      \lambda(\cdot)
    \end{bmatrix}=
    \begin{bmatrix}
      q(\cdot)\\
      Br(\cdot) + d(\cdot)\\
      \overline{s}\\
      \overline{\lambda}
    \end{bmatrix},
  \end{align}
  By \cite[theorem 10]{grune2020exponential}, under \Cref{ass:main}, the operator $M$ has a bounded inverse $M^{-1} : Y' \to X'$ with $X' := (L^2([0,T];\mathbb{R}^{n_s}))^2$ and $Y' := (L^2([0,T];\mathbb{R}^{n_s}))^2 \times \mathbb{R}^{n_s}\times \mathbb{R}^{n_s}$, where
  \begin{align*}
    (s^*(\cdot),\lambda^*(\cdot)) := M^{-1}\begin{bmatrix}
      q(\cdot)\\
      Br(\cdot) + d(\cdot)\\
      \overline{s}\\
      \overline{\lambda}
    \end{bmatrix}
  \end{align*}
  is a mild solution of \eqref{eqn:pont-reduced}; that is,
  \begin{align}
    \label{eqn:mild-2}
    \begin{aligned}
      \lambda^*(t)
      &=  \int_{t}^T e^{\Lambda^\top (\tau-t)} (C^\top Cs^* (\tau) - q(\tau)) d\tau + e^{\Lambda^\top(T-t)}  \overline{\lambda},\quad t\in[0,T]\\
      s^* (t)
      &= e^{\Lambda t} \overline{s}  + \int_{0}^t e^{\Lambda (t-\tau)} (BB^\top\lambda^*(t) + B r(t) +d(\tau)) d\tau,\quad t\in[0,T].
    \end{aligned}
  \end{align}
  We now consider
  $$H^{-1} = P_2 M^{-1} P_1 + P_3,$$ where $P_1 : Y \to Y'$, $P_2 : X' \to X$, and $P_3: Y \to X$ are given by
  \begin{align*}
    P_1(q,r,d,\overline{s},\overline{\lambda}) &= (q,Br+d,\overline{s},\overline{\lambda})\\
    P_2(s,\lambda) &= (s,B^\top\lambda,\lambda)\\
    P_3(q,r,d,\overline{s},\overline{\lambda}) &= (0,r,0),
  \end{align*}
  respectively.
  We observe that:
  \begin{align}\label{eqn:Hinv-def}
    H^{-1}(q,r,d,\overline{s},\overline{\lambda}) :=
    \begin{bmatrix}
      {s}^\star(\cdot)\\
      {u}^\star(\cdot)\\
      {\lambda}^\star(\cdot)\\
    \end{bmatrix} =
    \begin{bmatrix}
      s^*(\cdot)\\
      B^\top \lambda^*(\cdot) +d(\cdot)\\
      \lambda^*(\cdot)
    \end{bmatrix},
  \end{align}
  and $({s}^\star(\cdot),{u}^\star(\cdot),{\lambda}^\star(\cdot))$ satisfies \cref{eqn:mild} due to \cref{eqn:mild-2,eqn:Hinv-def}.
  Furthermore, one can confirm that $H^{-1}$ is bounded because the operator norms of $P_1,P_2$ and $P_3$ only depend on $B$, which is bounded by \Cref{ass:main}. Thus, $H^{-1}$ is a bounded, (mild) solution mapping for \cref{eqn:pont}.
\end{proof}
}

\subsection{Analytical Results}
We now state the main theorem of this section.
\begin{theorem}\label{thm:control}
  Under \Cref{ass:main}, for any given intervals $I_1,I_2\subseteq [0,T]$, $p(\cdot)\in L^2([0,T];\mathbb{R}^{n_z})$, and $\overline{s},\overline{\lambda}\in\mathbb{R}^{n_s}$, the following holds:
  \begin{align}\label{eqn:eds}
    \begin{aligned}[t]
      &\left[\int_{I_1} \|z^\star(t) \|_2^2 dt\right]^{1/2}
      \leq  8D\Bigg\{e^{\frac{-d(I_1,0)}{4D}}\|\overline{s}\|_2+ e^{\frac{-d(I_1,T)}{4D}}\|\overline{\lambda}\|_2\\
      &+e^{\frac{-d(I_1,I_2)}{4D}}\left[\int_{I_2} \| p(t)\|_2^2dt\right]^{1/2} +e^{\frac{-d(I_1,[0,T]\setminus I_2)}{4D}}\left[\int_{[0,T]\setminus I_2} \| p(t)\|_2^2dt\right]^{1/2} \Bigg\},
    \end{aligned}
  \end{align}
  where $z^\star(\cdot)$ is defined in \cref{eqn:zstar},
  $d(\cdot,\cdot)$ is the Euclidean distance on $[0,T]$, and 
  \begin{align}\label{eqn:const}
    D
    &:=\frac{L(1+2L)\left(\alpha^{2}+L^4(1+L)^2\right)^{1/2} }{\alpha^2  }  + \frac{L^2{(1+2L)^{2}}}{\alpha^2 }\max\left(1,\frac{L^{2} (1+L)^{2}}{(1+2L)^{2}}\right).
  \end{align}
\end{theorem}

The proof of \cref{thm:control} will be given later in this section. Here, we discuss its implications. \Cref{thm:control} shows that the solution mapping $H^{-1}$ exhibits exponential decay structure with respect to the time metric. In particular, the influence of the data located far from the time interval of interest $I_1$ decays exponentially with respect to the distance. This implies that one can approximate the solution over $I_1$ using only the data over a neighborhood of $I_1$ with an exponentially small error. Such exponential decay properties have been observed in various optimal control problems \cite{motee_optimal_2008,lin_perturbation-based_2021,shin_near-optimal_2023}, and \cref{thm:control} has recently gained interest in the control literature \cite{grune2020exponential,na2020exponential}. This property can be used to establish the convergence of overlapping decomposition algorithms \cite{na2020exponential} and to develop efficient non-uniform discretization methods \cite{shin_diffusing-horizon_2023,grune2020exponential}.

\WT{
  \cref{thm:control} is often interpreted as exponential decay of \textit{sensitivity}.
  This interpretation is relevant when one considers the perturbation analysis of \cref{eqn:ocp}.
  In particular, one can interpret $p(\cdot)$, $\overline{s}$, and $\overline{\lambda}$ as perturbations to the nominal data (which can be set to zero without loss of generality).
  Then, \cref{thm:control} shows how the perturbations in $p(\cdot)$, $\overline{s}$, and $\overline{\lambda}$ influence the outputs $z(\cdot)$.
  Specifically, the theorem states that if  $I_1, I_2 \subseteq [0,T]$ are well separated (so $d(I_1,I_2)$ is large), then perturbations in $p(\cdot)$ at $I_2$ have an exponentially decaying influence on the solutions $z(\cdot)$ at $I_1$.
  A similar relation holds for the exponentially decaying influence of the boundary terms $\overline{s}$ and $\overline{\lambda}$ on the solutions $z(\cdot)$ far from the boundary.
}

We now aim to show this by analyzing the discretization of the original problem. In particular, we apply a finite difference approximation scheme to \cref{eqn:pont}:
\begin{subequations}\label{eqn:ocp-disc}
  \begin{align}
    \frac{\lambda_k - \lambda_{k+1}}{h}- \Lambda^\top \lambda_{k+1} +C^\top C s_k
    &=  q_k ,\quad k\in\mathbb{Z}_{[0,N-1]}\\
    \lambda_N
    &= \overline{\lambda}\\
    u_k - B^\top \lambda_k,
    &= r_k,\quad k\in\mathbb{Z}_{[0,N-1]}  \\
    s_0
    &= \overline{s}\\
    \frac{s_{k} - s_{k-1}}{h} - \Lambda s_{k-1} - B u_{k-1}
    &= d_{k} ,\quad k\in\mathbb{Z}_{[1,N]}.
  \end{align}
\end{subequations}
where $h:=T/N$, $t_k:= hk$, $q_k:=q(t_k)$, $r_k:=r(t_k)$, and $d_k:=d(t_k)$. We rewrite \cref{eqn:ocp-disc} in a matrix form as follows:
\begin{align}\label{eqn:struct}
  \scriptsize
  \setlength\arraycolsep{1pt}
  \begin{bmatrix}
     C^\top C&&&&&&&&&\frac{1}{h}I&-\frac{1}{h}I-\Lambda^\top\\
       &C^\top C&&&&&&&&&\frac{1}{h}I&&\\
       &&\ddots&&&&&&&&&\ddots \\
       &&&C^\top C&&&&&&&&&\frac{1}{h}I&-\frac{1}{h}I-\Lambda^\top\\
       &&&&0&&&&&&&&&I\\
       &&&&&I&&&&&-B^\top\\
       &&&&&&\ddots&&&&&\ddots&\ddots\\
       &&&&&&&I&&&&&&-B^\top\\
    I&&&&\\
    -\frac{1}{h}I-\Lambda & \frac{1}{h}I&&&&-B \\
       &&\ddots&&&&\ddots\\
       &&&\frac{1}{h}I&&&\ddots\\
       &&&-\frac{1}{h}-\Lambda&\frac{1}{h}I&&&-B
  \end{bmatrix}
  \begin{bmatrix}
    s_0\\
    s_1\\
    \vdots\\
    s_{N-1}\\\
    s_{N}\\
    u_0\\
    \vdots\\
    u_{N-1}\\
    \lambda_0\\
    \lambda_1\\
    \vdots\\
    \lambda_{N-1}\\
    \lambda_N
  \end{bmatrix}
  =
  \begin{bmatrix}
    q_0\\
    q_1\\
    \vdots\\
    q_{N-1}\\
    \overline{\lambda}\\
    r_0\\
    \vdots\\
    r_{N-1}\\
    \overline{s}\\
    d_1\\
    \vdots\\
    d_{N-1}\\
    d_N
  \end{bmatrix}.
\end{align}
Here, one can observe that the matrix in \cref{eqn:struct} is $h$-banded, where the metric space is given by the discrete time horizon. 
In a block form, after the rescaling of the rows associated with $\overline{s}$ and $\overline{\lambda}$, we can rewrite \cref{eqn:struct} as:
\begin{align}\label{eqn:H-block}
  \underbrace{
  \begin{bmatrix}
    \bC_h^\top \bC_h&&\bLambda_h^\top\\
    &I & \bB_h^\top\\
    \bLambda_h & \bB_h\\
  \end{bmatrix}}_{\bH_h}
  \underbrace{
  \begin{bmatrix}
    \bs_h\\
    \bu_h\\
    \blambda_h\\
  \end{bmatrix}
  }_{\bz_h}
  =
  \underbrace{
  \begin{bmatrix}
    \bq_h\\
    \br_h\\
    \bd_h\\
  \end{bmatrix}
  }_{\bp_h},
\end{align}
where we let
\begin{align*}
  \bC_h
  &:=
    \begin{bmatrix}
      C\\
      &\ddots\\
      &&C&0
    \end{bmatrix};
  &&
     \bLambda_h:=
     \begin{bmatrix}
       \frac{1}{h}I\\
       -\frac{1}{h}I-\Lambda&\frac{1}{h}I\\
                            &\ddots&\ddots\\
                            &&-\frac{1}{h}I-\Lambda&\frac{1}{h}I\\
     \end{bmatrix};
  &&&
      \bB_h:=
      \begin{bmatrix}
        0\\
        B\\
        &\ddots\\
        &&B
      \end{bmatrix},
\end{align*}
$\bq_h:=(q_0, \cdots , q_{N-1}, \overline{\lambda}/h)$;  $\br_h:=(r_0, r_1, \cdots , r_{N-1})$;  and $\bd_h :=(\overline{s}/h, d_1, \cdots , d_N)$.

We will denote
\begin{align}\label{eqn:h-norm}
  \|\cdot\|_h:=h^{1/2}\|\cdot\|_2.
\end{align}
\WT{For $x(\cdot)\in C([0,T]; \mathbb{R}^{n_x})$ and $\bx_h:=[x(t_0);\cdots;x(t_N)]$, we have}
\begin{align*}
  \color{black}\|\bx_h\|_h = \left(\sum_{i=0}^N h\|x(t_i)\|^2\right)^{1/2}\rightarrow \left(\int_{0}^T \|x(t)\|^2dt\right)^{1/2}, 
\end{align*}
as $h\rightarrow 0$, since $\|x(\cdot)\|$ is continuous on $[0,T]$, and the Riemann sum of a continuous function converges to its integral.

In the following three lemmas, we will show that the discretized solution converges to the continuous solution, by showing consistency, numerical stability, and convergence. With this, later we will be able to apply the exponential decay result in \cref{cor:main} to the discretized solution to obtain the exponential decay result for the continuous solution.

% As a first step, we show that the discretization in \cref{eqn:ocp-disc} is consistent; meaning that the forward discretization mapping is asymptotically accurate.
\begin{lemma}[Consistency]\label{lem:cons}
  The following holds for all $p(\cdot)\in C([0,T];\mathbb{R}^{n_z})$ and $\overline{s},\overline{\lambda}\in\mathbb{R}^{n_s}$:
  \begin{align*}
    \lim_{h\rightarrow 0}\|\bH_{h}\bz^\star_h - \bp_h\|_h=0,
  \end{align*}
  where $\bz^\star_h:=(\bs^\star_h,\bu^\star_h,\blambda^\star_h)$, $\bs^\star_h:=(s^\star(t_0),\cdots,s^\star(t_N))$, $\bu^\star_h:=(u^\star(t_0),\cdots,u^\star(t_{N-1}))$, $\blambda^\star_h:=(\lambda^\star(t_0), \cdots,\lambda^\star(t_N))$, and $z^\star(\cdot):=(s^\star(\cdot), u^\star(\cdot), \lambda^\star(\cdot)) = H^{-1}(p(\cdot),\overline{s},\overline{\lambda})$.
\end{lemma}
\begin{proof}
  From \cref{eqn:mild}, one can see that
\begin{subequations}\label{eqn:zero}
  \begin{align}
    u^\star(t_k) - B^\top \lambda^\star(t_k)- r(t_k) &= 0,\quad \forall k=0,\cdots, N-1\\
    \frac{1}{h} s^\star(t_0) - \frac{1}{h}\overline{s} &= 0\\
    \frac{1}{h} \lambda^\star(t_N)  - \frac{1}{h}\overline{\lambda}  &= 0.
  \end{align}
\end{subequations}

Furthermore, since $p(\cdot)$ is continuous, we can observe from \cref{eqn:mild} that $z^\star(\cdot)$ is also continuously differentiable. Accordingly, each entry of $z^\star(\cdot)$ is continuously differentiable. 
Since each entry of $z^\star(\cdot)$ is continuously differentiable on a compact set, one can choose $h=h(\epsilon)>0$ for a given $\epsilon>0$ such that the following holds for all $t,t'\in[0,T]$ such that $|t-t'|\leq h$:
\begin{align*}
  |\dot{s}^\star[i](t)-\dot{s}^\star[i](t')| \leq \epsilon,\quad |\dot{\lambda}^\star[i](t)-\dot{\lambda}^\star[i](t')| \leq \epsilon.
\end{align*}
Furthermore, by Taylor's theorem, each entry of $\dot{s}^\star[i](\cdot)$ and $\dot{\lambda}^\star[i](\cdot)$ for $i=1,\cdots,n_s$ satisfies the following:
\begin{subequations}\label{eqn:taylor}
  \begin{align}
    \exists t'_{k,i}\in [t_{k},t_{k+1}]:& \quad \frac{s^\star[i](t_{k+1}) - s^\star[i](t_k)}{h}- \dot{s}^\star(t'_{k,i}) = 0 \\
    \exists t''_{k,i}\in [t_{k},t_{k+1}]:& \quad \frac{\lambda^\star[i](t_{k+1}) - \lambda^\star[i](t_k)}{h} -\dot{\lambda}^\star(t''_{k,i}) = 0
  \end{align}
\end{subequations}
Thus, we have the following for $k=0,\cdots,N-1$.
\begin{align}\label{eqn:res}
  \begin{aligned}[t]
    &\left\|\frac{s^\star(t_{k+1}) - s^\star(t_k)}{h} - \Lambda s^\star(t_k)  - Bu^\star(t_k) - d(t_k)\right\|_2\\
    &{=} \left(\sum_{i=1}^{n_s}\left|\dot{s}^\star[i](t'_{k})- \dot{s}^\star[i](t_{k})\right|^2\right)^{1/2}\\
    &\stackrel{\mathclap{\cref{eqn:pont},\cref{eqn:taylor}}}{\leq} \sqrt{n_s}\epsilon
      , \\
    &\left\|-\frac{\lambda^\star(t_{k+1}) - \lambda^\star(t_k)}{h} - \Lambda^\top \lambda^\star(t_{k+1})  + C^\top Cs^\star(t_k) - q(t_k)\right\|_2\\
    &{=} \left(\sum_{i=1}^{n_s}\left|\dot{\lambda}^\star[i](t''_{k})- \dot{\lambda}^\star[i](t_{k})\right|^2\right)^{1/2}\\
    &\stackrel{\mathclap{\cref{eqn:pont},\cref{eqn:taylor}}}{\leq}\sqrt{n_s} \epsilon        .
  \end{aligned}
\end{align}
Here, note that \cref{eqn:pont} holds due to the continuity of $p(\cdot)$.
By \cref{eqn:zero} and \cref{eqn:res}, 
\begin{align*}
  \|\bH_h \bz^\star_h - \bp_h\|_h
  &\leq\left[h\sum_{k=0}^{N-1}\left(\sqrt{n_s}\epsilon\right)^2+\left(\sqrt{n_s}\epsilon\right)^2\right]^{1/2}\\
  &\leq \sqrt{n_s}T\epsilon.
\end{align*}
Thus, we have that  $\|\bH_h\bz^\star_h - \bp_h\|_h\rightarrow 0$ as $h\rightarrow 0$.
\end{proof}

% Now, we show the numerical stability of the discretization in \cref{eqn:ocp-disc}. In particular, we show that the discretized solution mapping $\bH_h^{-1}$ is uniformly bounded above.

\begin{lemma}[Numerical Stability]\label{lem:stab}
  Under \Cref{ass:main}, the following holds:
  \begin{align*}
    \lim_{h\rightarrow 0} \|\bH_h^{-1}\|_2\leq D,
  \end{align*}
  where $D$ is defined in \cref{eqn:const}.
\end{lemma}

\begin{proof}
    We let $\bG_h:=
  \begin{bmatrix}
    \bC_h^\top \bC_h\\
    &I
  \end{bmatrix}$ and $\bF_h:=
  \begin{bmatrix}
    \bLambda_h& \bB_h
  \end{bmatrix}$. We now aim to show in order to apply \cref{thm:uniform} later:
  \begin{subequations}\label{eqn:uniform}
    \begin{align}
      \label{eqn:uniform-1}\bG_h& \preceq L^2 {I},\\
      \label{eqn:uniform-2}\lim_{h\rightarrow 0}\bF_h\bF^\top_h &\succeq \frac{\alpha^2}{L^2 (1+L)^2} {I}\\
      \label{eqn:uniform-3}\lim_{h\rightarrow 0}\bG + \bF_h^\top \bF_h &\succeq \frac{\alpha^2}{L^2 (1+2L)^2} {I}.
    \end{align}
  \end{subequations}
  We first show \cref{eqn:uniform-1}. The upper bound directly follows from the block-diagonal structure of $\bG_h$ and that $\|C\|$ is bounded above by $L\geq 1$ due to \cref{ass:main}.

  We now aim to show \cref{eqn:uniform-2}. Observe that:
  \begin{align*}
    \bF_h =
    \begin{bmatrix}
      \bLambda_h - \bB_h\bK_h& \bB_h\\
    \end{bmatrix}
    \begin{bmatrix}
      {I}\\
      \bK_{h}&{I}
    \end{bmatrix}, \text{ where }\bK_h:=
  \begin{bmatrix}
    K\\
    &\ddots\\
    &&K&0\\
  \end{bmatrix}.
  \end{align*}
  Here, $K$ derives from $(L,\alpha)$-stabilizability of $(\Lambda,B)$.
  Now one can see that:
  \begin{align}\label{eqn:control-1}
    \bF_h\bF_h^\top \succeq \left\|\left(\bLambda_h - \bB_h\bK_h\right)^{-1}\right\|_2^{-2}\left\|\begin{bmatrix}{I}&-\bK_h\\&{I}\end{bmatrix}\right\|_2^{-2} {I}.
  \end{align}
  From the $(L,\alpha)$-stabilizability of $(\Lambda,B)$,
  \begin{align}\label{eqn:control-2}
    \left\|\begin{bmatrix}{I}&-\bK_h\\&{I}\end{bmatrix}\right\|_2 \leq 1 + L.
  \end{align}
  Further, one can observe that
  \begin{align*}
    \bLambda_h-\bB_h\bK_h=
    \frac{1}{h}
    \begin{bmatrix}
      {I}\\
      -{I}-h\Phi&{I}\\
      &&\ddots\\
      &&-{I}-h\Phi&{I}\\
    \end{bmatrix},
  \end{align*}
  where $\Phi:=\Lambda-BK$, and
  \begin{align*}
    (\bLambda_h-\bB_h\bK_h)^{-1} = h
    \begin{bmatrix}
      {I}\\
      {I}+h\Phi&{I}\\
      ({I}+h\Phi)^2&{I}+h\Phi&{I}\\
      \vdots&\vdots&\ddots&\ddots\\
      ({I}+h\Phi)^{N-1}&({I}+h\Phi)^{N-2}&\cdots&{I}+h\Phi&{I}\\
    \end{bmatrix}.
  \end{align*}
  Thus,
  \begin{align}\label{eqn:oppsum}
    \begin{aligned}[t]
      &\left\|(\bLambda_h-\bB_h\bK_h)^{-1}\right\|_2\\
      &\leq h\left\{
        \left\| I \right\|_2
        +
        \left\|
        \begin{bmatrix}
          0\\
          {I}+h\Phi&0\\
                   &{I}+h\Phi&0\\
                   &&\ddots&\ddots\\
                   &&&{I}+h\Phi&0\\
        \end{bmatrix}\right\|_2
        +
        \cdots
        +
        \left\|
        \begin{bmatrix}
          \\
          \\
          \\
          \\
          ({I}+h\Phi)^{N-1}&&&&&\\
        \end{bmatrix}\right\|_2
        \right\}\\
      &\leq h\sum_{k=0}^{N-1}\left\| (I+h\Phi)^k\right\|_2.
    \end{aligned}
  \end{align}
  We observe that $e^{\Phi t}$ is twice continuously differentiable and has uniformly bounded second-order derivatives on $[0,T]$. Thus, by Taylor's theorem, we have that there exists $M>0$ such that
  \begin{align}\label{eqn:c2small}
    \|e^{\Phi h} - I - h\Phi\|_2 \leq Mh^2.
  \end{align}
  We observe for sufficiently small $h$,
  \begin{align}\label{eqn:oppsum-2}
    \begin{aligned}[t]
      \left\| (I+h\Phi)^k\right\|_2
      &\;\;\leq\;\; \left\| e^{\Phi hk}\right\|_2 + \left\| e^{\Phi hk} - (I+h\Phi)^k\right\|_2\\
      &\stackrel{\mathclap{\text{\cref{ass:main}}}}{\;\;\leq\;\;}  L e^{-\alpha hk} +\left\| (I + h\Phi + (e^{\Phi h} - I - h\Phi))^{k} - (I + h\Phi)^k\right\|_2\\
      &\;\;\leq\;\;  L e^{-\alpha hk} + \sum_{i=1}^k {k \choose i}  \|I + h\Phi\|_2^{k-i} (Mh^2)^i\\
      &\;\;\leq\;\;  L e^{-\alpha hk} + \sum_{i=1}^k (1+\|\Phi\|_2h)^N (kMh^2)^i\\
      &\;\;\leq\;\;  L e^{-\alpha hk} + 2{(1+\|\Phi\|_2h)^N (kMh^2)}.
    \end{aligned}
  \end{align}
  Here,
  the third inequality follows from the binomial theorem and \cref{eqn:c2small};
  the fourth inequality follows from  $k-i\leq N$, and ${k \choose i}\leq k^i$;
  and the last inequality follows from $kMh^2<1/2$ for sufficiently small $h$. 
  Therefore, by combining \cref{eqn:oppsum} and \cref{eqn:oppsum-2}, we obtain
  \begin{align}\label{eqn:control-3}
    \left\|(\bLambda_h-\bB_h\bK_h)^{-1}\right\|_2
    &\leq h\sum_{k=0}^{N-1}\left(L e^{-\alpha hk} + 2{(1+\|\Phi\|_2h)^N (kMh^2)}\right)\\
    &\leq \frac{Lh}{1-e^{-\alpha h}} +M(1+\|\Phi\|_2h)^N(N^2-N)h^3   \rightarrow \frac{L}{\alpha} + M \cdot e^{\|\Phi\|_2 T} \cdot 0 
  \end{align}
  as $h\rightarrow 0$. Therefore, by \cref{eqn:control-1}, \cref{eqn:control-2}, and \cref{eqn:control-3}, we obtain \cref{eqn:uniform-2}.

  Now we aim to show  \cref{eqn:uniform-3}. We observe that:
  \begin{align*}
    \bG_h + \bF_h^\top\bF_h
    &=\begin{bmatrix}
      \bC_h^\top & &\bLambda^\top_h\\
      &I & \bB^\top_h
    \end{bmatrix}
    \begin{bmatrix}
      \bC^\top_h& &\bLambda^\top_h\\
      &I & \bB^\top_h
    \end{bmatrix}^\top,
  \end{align*}
  and
  \begin{align}\label{eqn:obs-1}
    \begin{bmatrix}
      \bC^\top_h& &\bLambda^\top_h\\
      &I & \bB^\top_h
    \end{bmatrix}
    =
    \begin{bmatrix}
      \bC^\top_h & &\bLambda^\top_h - \bC^\top_h\bK^\top_h\\
      &I &
    \end{bmatrix}
    \begin{bmatrix}
      {I}&&\bK_{h}^\top\\
      &{I}&\bB^\top_h\\
         &&{I}
    \end{bmatrix}, \text{ where }\bK_h:=
    \begin{bmatrix}
      0\\
      K\\
      &\ddots\\
      &&K\\
    \end{bmatrix},
  \end{align}
  where $K$ is derived from $(L,\alpha)$-detectability of $(\Lambda,C)$.
  We can derive the following:
  \begin{align*}
    &\bG_h+\bF_h^\top\bF_h \succeq \min\left(\left\|\left(\bLambda_h - \bK_h\bC_h\right)^{-1}\right\|_2^{-2},1\right)\left\|
      \begin{bmatrix}{I}\\&{I}\\-\bK_h&-\bB_h&{I}\end{bmatrix}
    \right\|_2^{-2} {I}.
  \end{align*}
  From the $(L,\alpha)$-detectability of $(\Lambda,C)$, one can easily see that
  \begin{align}\label{eqn:obs-2}
    \left\|
      \begin{bmatrix}{I}\\&{I}\\-\bK_h&-\bB_h&{I}\end{bmatrix}
    \right\|_2\leq  1+ 2L.
  \end{align}
  Similarly, to $\bLambda_h-\bB_h\bK_h$ case, one can show that
  \begin{align}\label{eqn:obs-3}
    \lim_{h\rightarrow 0}\left\|(\bLambda_h-\bK_h\bC_h)^{-1}\right\|_2 \leq \frac{L}{\alpha}.
  \end{align}
  Therefore, by \cref{eqn:obs-1}, \cref{eqn:obs-2}, and \cref{eqn:obs-3}, we can obtain \cref{eqn:uniform-3}.
  Finally, due to \cref{thm:uniform} and \cref{eqn:uniform}, we have that
  \begin{align*}
    \|\bH^{-1}_h\|_2
    &\leq \left(\frac{1+L^2/(\frac{\alpha^2}{L^2 (1+L)^2})}{\frac{\alpha^2}{L^2 (1+2L)^2}}\right)^{1/2}  + \max\left(\frac{1}{\frac{\alpha^2}{L^2 (1+2L)^2}},\frac{L^2}{\frac{\alpha^2}{L^2 (1+L)^2}}\right)  \\
    &\leq \frac{L(1+2L)\left(\alpha^{2}+L^4(1+L)^2\right)^{1/2} }{\alpha^2 }  + \frac{L^2{(1+2L)^{2}}}{\alpha^2 }\max\left(1,\frac{L^{2} (1+L)^{2}}{(1+2L)^{2}}\right).
      % \tag*{\qedhere}
  \end{align*}
\end{proof}

% Finally, we show the convergence of the discretized solution to the continuous solution. This is a direct consequence of consistency and numerical stability.

\begin{lemma}[Convergence]\label{lem:conv}
  Under \Cref{ass:main} and $p(\cdot)\in C([0,T];\mathbb{R}^{n_z})$, the following holds:
  \begin{align*}
    \lim_{h\rightarrow 0}\|\bz^\star_h - \bH^{-1}_h \bp_h\|_h = 0.
  \end{align*}
\end{lemma}
\begin{proof}
    Observe:
  \begin{align*}
    \|\bz^\star_h - \bH^{-1}_h \bp_h\| \leq \|\bH^{-1}_h\|_{h\rightarrow h}\|\bH_{h}\bz^\star_h - \bp_h\|_h.
  \end{align*}
  We observe that as an operator norm, $\|\cdot\|_{h\rightarrow h} = \|\cdot\|_2$. By taking $h\rightarrow 0$, and from the consistency (\cref{lem:cons}) and stability (\cref{lem:stab}), we obtain the desired result.
\end{proof}

From \cref{lem:stab} and the structure of $\bH_h$  in \cref{eqn:H-block}, one can verify that the exponential decay bound in \cref{cor:main} for the discretized problem in \cref{eqn:ocp-disc} is uniformly bounded.

Now, we are ready to prove \cref{thm:control}. We apply the exponential decay result in \cref{cor:main} to the discretized solution $\bz^\star$ to show the exponential decay in the discrete regime, and we use the asymptotic convergence result in \cref{lem:conv} to map this result to continuous time domain. We generalize the result for $p(\cdot)\in C([0,T];\mathbb{R}^{n_z})$ into $p(\cdot)\in L^2([0,T];\mathbb{R}^{n_z})$ by using the boundedness of $H^{-1}$ \WT{(recall \Cref{thm:bdd-sol-exists})} and the density of $C([0,T];\mathbb{Z}^{n_z})$ within the Hilbert space $L^2([0,T];\mathbb{R}^{n_z})$.
Here, we will see that our tight bound in \cref{thm:main} is essential to prove \cref{thm:control}; if the loose bounds in \cref{prop:demko,prop:shin} are used, the bound diverges upon mesh refinement and one cannot establish the asymptotic bound.

\begin{proof}[Proof of \cref{thm:control}]
  First, we consider a special case, where $p(\cdot)\in C([0,T];\mathbb{R}^{n_z})$. From the triangle inequality and the submultiplicativity of matrix norms, we have the following:
  \begin{align}\label{eqn:breakdown}
    \begin{aligned}[t] 
      &\left[\int_{I_1} \|z^\star(t) \|^2 dt\right]^{1/2}
        \leq
        \underbrace{\left(\left[\int_{I_1} \|z^\star(t) \|^2 dt\right]^{1/2} - \|  (\bz^\star_h)_{I_1}\|_h \right)}_{\text{term 1}}\\
      & +
        \underbrace{\| \bz^\star_h - \bH^{-1}_h \bp_h\|_h}_{\text{term 2}}
        +\underbrace{\|(\bH^{-1}_h)_{I_1,I_2\setminus I_{\bar s}\setminus I_{\bar\lambda}}\|_{h\leftarrow h }\| (\bp_h)_{I_2 \setminus I_{\bar s}\setminus I_{\bar\lambda}}\|_h}_{\text{term 3}}
      \\
      &
        +\underbrace{\|(\bH^{-1}_h)_{I_1,[0,T]\setminus I_2\setminus I_{\bar s}\setminus I_{\bar\lambda}}\|_{h\leftarrow h }\| (\bp_h)_{[0,T]\setminus I_2 \setminus I_{\bar s}\setminus I_{\bar\lambda}}\|_h}_{\text{term 4}}
      \\
      &
        + \underbrace{\|(\bH^{-1}_h)_{I_1,I_{\bar s}}\|_{h\leftarrow h }\| (\bp_h)_{I_{\bar s}}\|_h}_{\text{term 5}}
        +\underbrace{ \|(\bH^{-1}_h)_{I_1,I_{\bar\lambda}}\|_{h\leftarrow h }\| (\bp_h)_{I_{\bar\lambda}}\|_h.}_{\text{term 6}}
    \end{aligned}
  \end{align}
  where $(\cdot)_{I}$ denotes the restriction of a vector $(\cdot)$ to the space whose time index is within $I\subseteq [0,T]$, and $(\cdot)_{I,J}$ denotes the restriction of a matrix $(\cdot)$  to the row and column spaces whose time indexes are within $I\subseteq[0,T]$ and $J\subseteq [0,T]$, respectively. We denote by $(\cdot)_{I_{\overline{s}}}$ and $(\cdot)_{I_{\overline{\lambda}}}$ the restriction of the argument to the indexes associated with $\overline{s}$ and $\overline{\lambda}$, respectively. Furthermore, we let $(\cdot)_{I\setminus I_{\overline{s}}\setminus I_{\overline{\lambda}}}$ denote the vector whose time index is within $I\subseteq [0,T]$, but excluding the indexes associated with $\overline{s},\overline\lambda$. A similar convention is used for the for matrix arguments involving $I_{\overline{s}}$, $I_{\overline{\lambda}}$, and $I\setminus I_{\overline{s}}\setminus I_{\overline\lambda}$.

  Since $\bz^\star(\cdot)$ is continuous, the Riemann sum of $\|z^\star(\cdot)\|$ over $I_1$ converges to its integral, and thus, we have that term 1 $\rightarrow 0$ as $h\rightarrow 0$. Further, by \cref{lem:conv}, we have that term 2 $\rightarrow 0$ as $h\rightarrow 0$.

  Now we aim to analyze terms 3--6.  First, observe that
  \begin{align*}
    \| (\bp_h)_{I_{\bar s}}\|_h
    &=\|\overline{s}\|\\
    \| (\bp_h)_{I_{\bar \lambda}}\|_h
    &=\|\overline{\lambda}\|\\
    \lim_{h\rightarrow 0}\| (\bp_h)_{I_2 \setminus I_{\bar s}\setminus I_{\bar\lambda}}\|_h
    &=\left[\int_{I_2} \|p(t)\|^2 dt\right]^{1/2}\\
    \lim_{h\rightarrow 0}\| (\bp_h)_{[0,T]\setminus I_2 \setminus I_{\bar s}\setminus I_{\bar\lambda}}\|_h
    &=\left[\int_{[0,T]\setminus I_2} \|p(t)\|^2 dt\right]^{1/2}.
  \end{align*}
  Here, the last two equalities follow from the convergence or Riemann sum to its integral for continuous functions.
  Now, due to \cref{lem:stab}, for sufficiently small $h$, we have that $\|\bH^{-1}_h\| \leq 2D$. Further, by inspecting the structure of $\bH_h$ in \cref{eqn:H-block}, one can see that
  \begin{align*}
    \|\bH_h\|
    &\leq
      \left\|
      \begin{bmatrix}
        \bC_h^\top \bC_h&&\bLambda_h^\top\\
                        &I & \bB_h^\top\\
        \bLambda_h & \bB_h\\
      \end{bmatrix}
      \right\|_2\\
    &\leq
      \left\|
      \begin{bmatrix}
        \bC_h^\top \bC_h&&\\
                        & & \bB_h^\top\\
         & \bB_h\\
      \end{bmatrix}
      \right\|_2
      +
      \left\|
      \begin{bmatrix}
        &&\bLambda_h^\top\\
                        &I & \\
        \bLambda_h &\\
      \end{bmatrix}
      \right\|_2\\
    &\leq
      \max(L^2,L)
      +
      \max(L+\frac{2}{h},1)
    \\
    &\leq
      L(L+1)+2/h
  \end{align*}
  Therefore, we have that $\sigma_{>0}(\bH_h)\in [a,b]$, where $a=1/2D$ and $b=L(L+1)+2/h$. Now we aim to analyze
  \begin{align*}
  f_2(d/h; a, b)
  &\leq \dfrac{4 (b+a)^{3/2}}{ab(b-a)^{1/2}}\left(\dfrac{b/a-1}{b/a+1}\right)^{ \frac{d/h -1}{2}}.
\end{align*}
We observe that 
\begin{align}\label{eqn:this-is-where-it-breaks}
  \lim_{h\rightarrow 0}\frac{4(b+a)^{3/2}}{ab(b-a)^{1/2}}=\WT{8D} 
\end{align}
and
\begin{align*}
  \lim_{h\rightarrow 0}\left(\dfrac{b/a-1}{b/a+1}\right)^{ \frac{d/h -1}{2}}
  &=\lim_{h\rightarrow 0}\left(\dfrac{2D L(L+1) + 4D/h - 1}{2D L(L+1) + 4D/h + 1}\right)^{ \frac{d/h -1}{2}}\\
  &=\lim_{h\rightarrow 0}\left(\dfrac{1+(L(L+1)/2-1/4D)h}{1+(L(L+1)/2+1/4D)h}\right)^{ \frac{d/h -1}{2}}\\
  &= \left(\frac{e^{L(L+1)/2-1/4D}}{e^{L(L+1)/2+1/4D}}\right)^{d/2}\\
  &= e^{-d/4D}.
\end{align*}
Therefore, for sufficiently small $h>0$,
\begin{align}\label{eqn:f2}
  f_2(d/h;a, b) \leq 8D e^{-d/4D}.
\end{align}
  We observe that by \cref{eqn:f2}, \cref{cor:main}, and the fact that $\bH_h$ is $h$-banded (can be confirmed from the structure of $\bH_h$ in \cref{eqn:struct}),
\begin{align*}
  \lim_{h\rightarrow 0}\|(\bH^{-1}_h)_{I_1,I_{\bar s}}\| &\leq \WT{8D} e^{-d(I_1,0)/4D}\\
  \lim_{h\rightarrow 0}\|(\bH^{-1}_h)_{I_1,I_{\bar \lambda}}\|&\leq  \WT{8D} e^{-d(I_1,T)/4D}\\
  \lim_{h\rightarrow 0}\|(\bH^{-1}_h)_{I_1,I_2\setminus I_{\bar s}\setminus I_{\bar \lambda}}\|  &\leq \WT{8D}  e^{-d(I_1,I_2)/4D}\\
\lim_{h\rightarrow 0}\|(\bH^{-1}_h)_{I_1,[0,T]\setminus I_2\setminus I_{\bar s}\setminus I_{\bar \lambda}}\|  &\leq \WT{8D} e^{-d(I_1,[0,T]\setminus I_2)/4D}.\end{align*}
Thus, by taking a limit $h\rightarrow 0$ in \cref{eqn:breakdown}, we obtain the desired result for $p(\cdot)\in C([0,T];\mathbb{R}^{n_z})$ case.

Now, we consider $p(\cdot)\in L^2([0,T];\mathbb{R}^{n_z})$ case. We observe that there exists a sequence of continuous functions $\{p_k(\cdot)\}_{k=0}^\infty$ such that $p_k(\cdot)\rightarrow p(\cdot)$ in $L^2([0,T];\mathbb{R}^{n_z})$ since $C([0,T];\mathbb{R}^{n})$ is dense in $L^2([0,T];\mathbb{R}^{n_z})$ due to \cite[lemma A.5.19]{curtain2012introduction}. Since $H^{-1}$ is bounded by \Cref{thm:bdd-sol-exists}, and bounded linear operators are continuous, we have that
\begin{align*}
  z_k(\cdot):=H^{-1} (p_k(\cdot),\overline{s},\overline{\lambda})\rightarrow z^\star(\cdot) = H^{-1} (p(\cdot),\overline{s},\overline{\lambda}),\quad \text{ as } k\rightarrow \infty.
\end{align*}
Since we have \cref{eqn:eds} for each $\{p_k(\cdot)\}_{k=0}^\infty$, we can obtain the desired result by taking $k\rightarrow \infty$.
\end{proof}

Lastly, we reiterate the importance of the tight bounds in \cref{thm:main,thm:uniform} in proving \cref{thm:control}. 
Note that if the bounds in \cref{prop:demko,prop:shin} are used, the limit in \cref{eqn:this-is-where-it-breaks} becomes unbounded. Therefore, with the existing bounds, one cannot obtain the asymptotic bound in \cref{thm:control}. We also see that the bounds in \cref{thm:uniform} play a crucial role in that the lower bound of the singular values $a$ is independent of $h$. This allows the decay bounds not to diverge as $h\rightarrow 0$.

\WT{
\section{Numerical Results}\label{sec:num}
We now aim to numerically validate \cref{thm:control}.}
Consider the problem in \cref{eqn:ocp} with $n_s=1$, $n_u=1$, $T=10$, and $\Lambda=1.1$. We analyze two cases: (i) regular case with $C=1$ and $B=1$, and (ii) near-singular case with $C=10^{-3}$ and $B=10^{-6}$. We also consider two cases of perturbations: (i) boundary perturbation with $\overline{s}=1$, $\overline{\lambda}=1$, and $q(t) = r(t) = d(t) \equiv 0$, and (ii) middle perturbation with $q(t) = r(t) = d(t) = \boldsymbol{1}_{[4,6]}$ and $\overline{s}=\overline{\lambda}=0$. We solved the discretized problems \cref{eqn:ocp-disc} with different mesh sizes and visualized the solutions $s(\cdot)$ and $s(\cdot)$ over $[0,1]$ in \cref{fig:results}. In the regular cases, perturbation effects decay exponentially, but this decay is not observed in the near-singular cases. The decay rate converges to a common rate as the mesh size is refined, validating our theoretical development. The source code to reproduce these results is available online \cite{noauthor_sshin23banded-approx-code_nodate}.

\begin{figure}[h]
  \centering
  \begin{tikzpicture}[scale=.8]
    \node[scale=.45] at (0,0) {\includegraphics{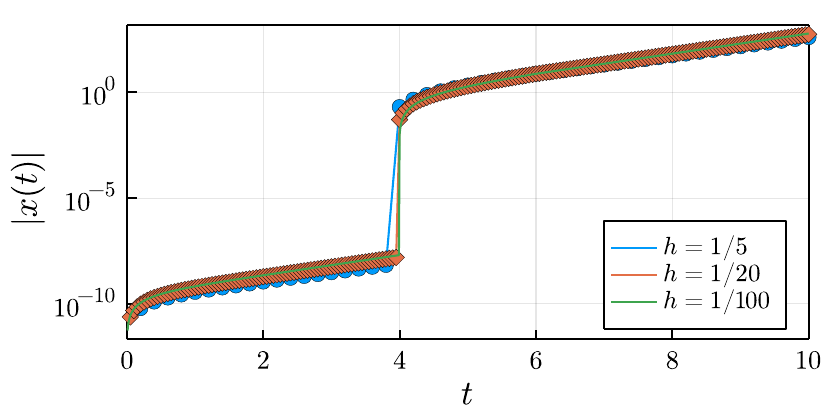}};
    \node[scale=.45] at (8,0) {\includegraphics{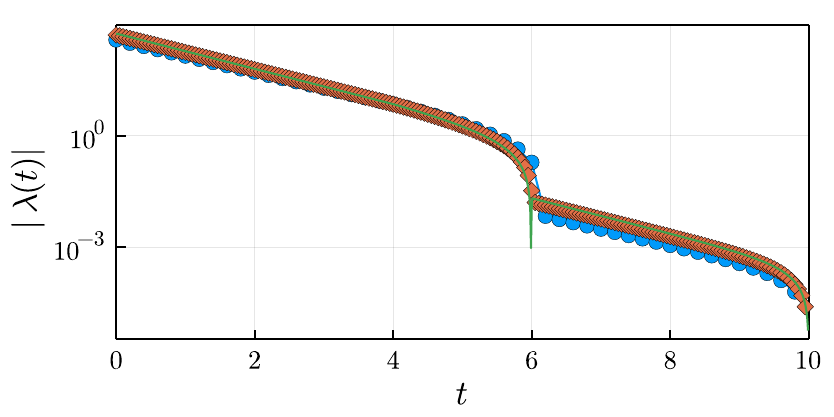}};
    \node[scale=.45] at (0,4) {\includegraphics{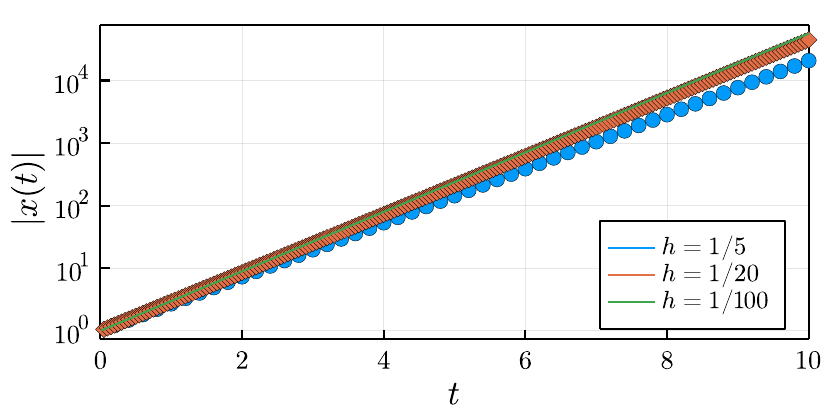}};
    \node[scale=.45] at (8,4) {\includegraphics{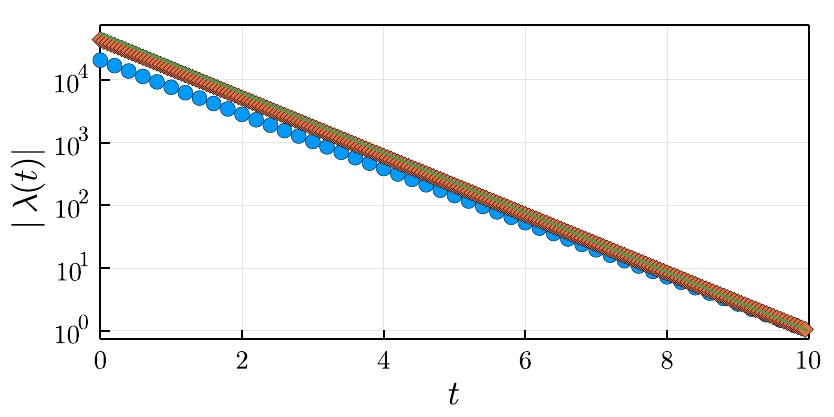}};
    \node[scale=.45] at (0,8) {\includegraphics{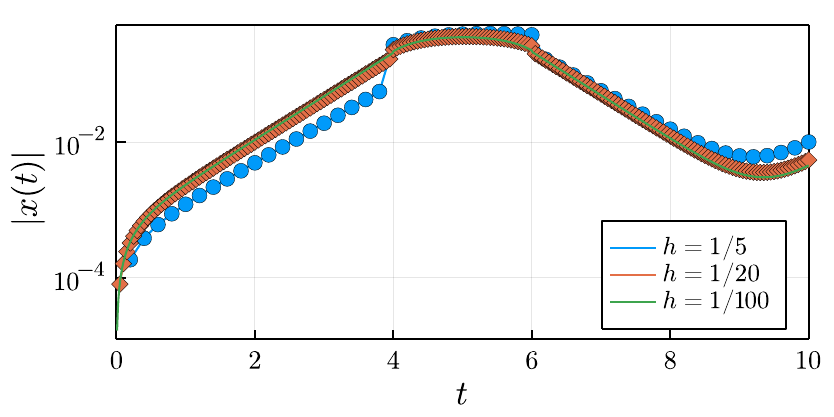}};
    \node[scale=.45] at (8,8) {\includegraphics{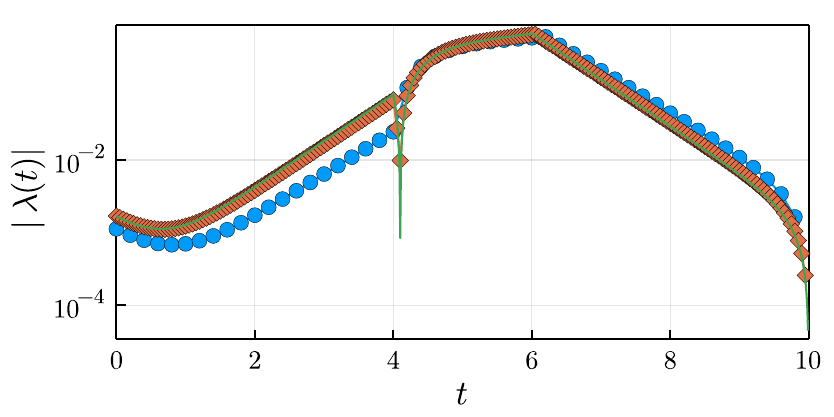}};
    \node[scale=.45] at (0,12) {\includegraphics{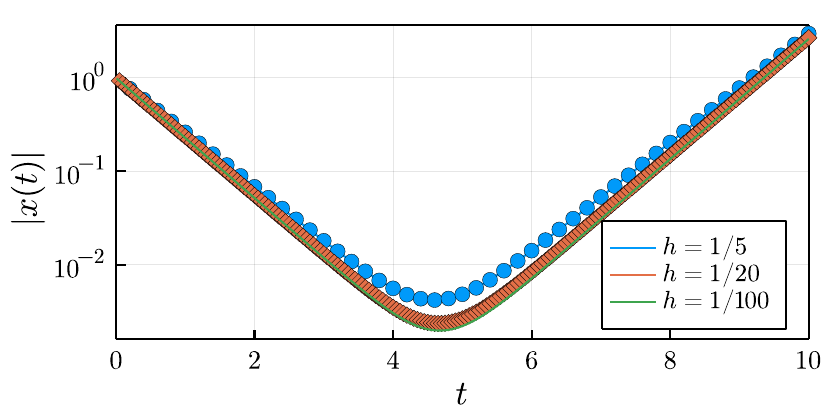}};
    \node[scale=.45] at (8,12) {\includegraphics{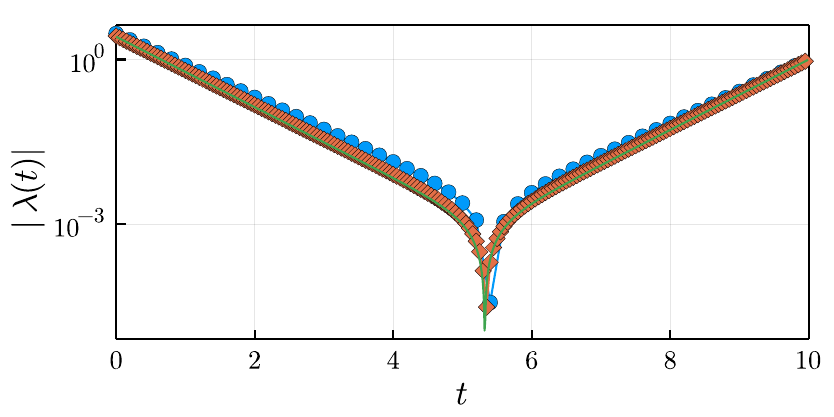}};
  \end{tikzpicture}
  \caption{Numerical validation of \cref{thm:main}. Top: regular, boundary perturbation. Second from the top: regular, middle perturbation. Third from the top: near-singular, boundary perturbation. Bottom: near-singular, middle perturbation.}\label{fig:results}
\end{figure}

\section{Conclusions and Future Work}\label{sec:concl}

We have presented an improved approximation bound for the Moore-Penrose inverse of banded systems. Our improved bounds are particularly useful for analyzing saddle point systems resulting from the discretization of continuous domain problems. Specifically, our bound does not diverge as the mesh size is refined; we have demonstrated this by showing that solution mappings for optimal control problems exhibit the exponentially decaying structure, given mild assumptions. We anticipate that our findings can be extended to various problem domains, including optimization problems constrained by partial differential equations.
% While the existing works in the literature \cite{demko1984decay,bickel2012approximating} have shown the approximation bounds for bounded operators on $\ell^2$ sequence spaces. However, such generalization is not trivial for more general settings considered in this work, where the operator can be rectangular and/or have a rank deficiency, and the singular value decomposition is utilized within the proof. Thus, we leave such a generalization as future work.
In the future, we are interested in analyzing the banded approximation of linear operators on a suitable class of Banach and Hilbert spaces. 

\section*{Acknowledgment}
This material was based upon work
supported by the U.S. Department of Energy, Office of Science,
Office of Advanced Scientific Computing Research (ASCR) under
Contract DE-AC02-06CH11347. Wallace Tan is supported by the MathWorks Fellowship. 

\bibliographystyle{siamplain}
\bibliography{references}

\end{document}